\documentclass[3p, times]{elsarticle}

\usepackage{amsmath}
\usepackage{amssymb}
\usepackage{amsthm}
\usepackage{hyperref}
\usepackage{mathrsfs}
\usepackage{multirow}
\usepackage{booktabs}
\usepackage{siunitx}
\usepackage[ruled,linesnumbered]{algorithm2e}
\usepackage{graphicx}
\usepackage{subcaption}
\usepackage{xcolor}
\newtheorem{Proposition}{Proposition}

\newtheorem{theorem}{Theorem}[section]
\newtheorem{assumption}{Assumption}[section]

\journal{}
\biboptions{sort&compress}
\begin{document}
\begin{frontmatter}

\title{Boundary neuron method for solving partial differential equations}

\author[JLU]{Ye Lin}\ead{linye21@mails.jlu.edu.cn}
\affiliation[JLU]{
organization={School of Mathematics, Jilin University}, 
addressline={2699 Qianjin Street}, 
city={Changchun},
postcode={130012}, 
state={Jilin},
country={China}.}

\author[JLU]{Wentao Liu}\ead{wtliu24@mails.jlu.edu.cn}

\author[TXST]{Young Ju Lee}
\ead{yjlee@txstate.edu}
\affiliation[TXST]{
organization={Department of Mathematics, Texas State University}, 
addressline={601 University Drive}, 
city={San Marcos},
postcode={78666}, 
state={Texas},
country={U.S.A.}.}

\author[JLU]{\texorpdfstring{Jiwei Jia\corref{cor1}}{Jiwei Jia}}\ead{jiajiwei@jlu.edu.cn}
\cortext[cor1]{Corresponding author.}

\begin{abstract}
We propose a boundary neuron method with random features (BNM-RF) for solving partial differential equations. The method approximates the unknown boundary function by a shallow network within the boundary integral formulation. With randomly sampled and fixed hidden parameters, the computation reduces to a linear least squares problem for the output coefficients, which avoids gradient based nonconvex optimization. This construction retains the dimensionality reduction of boundary integral equations and the linear solution structure of the random feature method. For elliptic problems, we establish convergence analysis by combining kernel-based method with random feature approximation, and obtain error bounds on both the boundary and the interior solution. Numerical experiments on Laplace and Helmholtz problems, including interior and exterior cases, show that the proposed method achieves competitive accuracy relative to the boundary element method and favorable performance relative to boundary integral neural networks in the tested settings with only few neurons. Overall, the proposed method provides a practical framework for combining boundary integral equations with neural network for problems on complex geometries and unbounded domains.
\end{abstract}


\end{frontmatter}

\section{Introduction}
\label{intro}

Solving partial differential equations using boundary integral formulations plays an important role in scientific computing and engineering simulations, especially in acoustics, computational electromagnetics, and elasticity mechanics \cite{kirkup2019acoustics,chew2022integral,banerjee1983elastostatics}. For such problems, boundary integral equations provide an attractive reformulation. By transferring the unknowns from the domain to the boundary, they reduce the dimension of the computation and are naturally suited to exterior problems. The boundary element method (BEM) is a classical discretization of this framework and has been widely used in practice \cite{atkinson1997numerical}. However, for problems with complex geometries or boundary conditions, the boundary element method often requires a fine boundary discretization to achieve high accuracy. Such refined discretizations usually lead to dense linear systems, which in turn result in substantial computational and storage costs.

In recent years, deep neural network methods for partial differential equations have attracted considerable attention \cite{raissi2019physics,karniadakis2021physics}. In particular, a growing body of work have combined neural networks with boundary integral formulations or restricted the approximation to boundary quantities. BINN \cite{binn_Sun2023} is an early example of the latter idea, it approximates only the boundary unknowns and uses the residuals of the boundary integral equations in the loss function, thereby reducing the problem dimension by one and making the method well suited to infinite, semi infinite, and heterogeneous domain problems. BINet \cite{binet_Lin2023}, by contrast, builds the solution directly from boundary layer potentials, so that the governing PDE is satisfied automatically and only the boundary conditions enter the loss. BI-GreenNet \cite{bigreennet_Lin2023} further develops this representation based idea by shifting the learning target from the solution itself to Green’s functions. Related ideas have also been used for two dimensional elastostatic and piezoelectric problems in \cite{zhang2024boundaryintegrated}, where exact boundary integral equations confine the unknowns to the boundary and replace differential operators with integral operators.
Despite these advantages, existing boundary integral network approaches still rely mainly on gradient based training, and their performance may therefore remain sensitive to optimization efficiency and training robustness.

In addition to fully trainable neural networks, shallow network methods have also attracted increasing attention in scientific computing. For PDEs, two representative examples are finite neuron method (FNM) \cite{xu2020fnm} and the random feature method (RFM) \cite{rahimi2007random,liu2021random}. FNM provides a shallow neural network framework with approximation properties similar to those of finite element methods and rigorous convergence analysis \cite{xu2020fnm}, and subsequent work has further developed efficient training strategies such as neuron wise subspace correction methods \cite{park2025fnm}. RFM, in contrast, fixes the hidden parameters in advance and determines the output coefficients by solving a linear least squares problem. This linear solution structure avoids repeated nonconvex optimization while retaining the approximation power of shallow neural networks. RFM has been applied to a range of PDE settings, including elliptic, time dependent, interface, and diffusion problems, as well as high precision regimes \cite{chen2022bridging,chen2024optimization,chen2023random,chi2024random,mei2025solving,chen2024high}. Its least squares structure has also motivated improved optimization and iterative solvers \cite{chen2024optimization,   chen2024high}, while related theoretical work has addressed random features, as well as kernel and Gaussian process based error analysis \cite{liu2021random, lanthaler2023error, liao2025solving, batlle2025error}. Despite this progress, the use of shallow networks with random features in boundary integral equations remains largely unexplored. In particular, existing theoretical results for RFM are mainly developed for strong form formulations, whereas the boundary integral setting requires approximation and stability analysis directly on the boundary. 
At the same time, existing neural methods based on boundary integral equations still rely mainly on gradient based optimization. This creates a gap between the advantages of boundary integral formulations and the linear solution structure of RFM.

In this paper, we propose the boundary neuron method with random features (BNM-RF) for solving boundary value problems through boundary integral equations. The proposed method approximates the unknown boundary function within the boundary integral formulation and determines the corresponding coefficients through a linear least squares problem. In this way, it combines the dimension reduction of boundary integral equations with the linear solution procedure of shallow networks with random features. Unlike existing boundary integral network approaches, which are trained by gradient based optimization, the proposed framework determines the boundary unknowns through a single linear solve and avoids iterative training in the main solution stage.

The contributions of this work are threefold. First, we formulate BNM-RF as a boundary integral solver with random features and derive the corresponding least squares system for the boundary unknowns. Second, for elliptic problems, we extend the random feature error analysis from strong form equations to the boundary integral setting and obtain error bounds for both the boundary function and the interior solution. Third, we validate the proposed method on Laplace and Helmholtz problems, including interior and exterior cases. The numerical results show that BNM-RF is highly competitive with BINN and the classical BEM in the reported two dimensional tests, while the three dimensional acoustic scattering example also identifies singular integral treatment and geometric approximation as the main issues for further improvement.

The remainder of the paper is organized as follows. Section \ref{sec:boundary_formulation} reviews the boundary integral formulation and related background. Section \ref{sec:BNM} presents the proposed BNM-RF method. Section \ref{sec:analysis} develops the convergence analysis. Section \ref{sec:result} reports the numerical results. Finally, Section \ref{sec:conclusion} concludes the paper.

\section{Background}
\label{sec:boundary_formulation}

\subsection{Problem description}
\label{sec:problem-description}
Let $\Omega \subset \mathbb{R}^d$ be a bounded domain with closure  $\overline{\Omega}$ and exterior $\Omega^c = \mathbb{R}^d \setminus \overline{\Omega}$. We consider the following boundary value problem (BVP) for a homogeneous PDE:
\begin{equation}
\label{eq:harmonic_pde}
\begin{split}
\mathcal{L} u(\mathbf{x}) &= 0, \quad \mathbf{x} \in \Omega \text{ or } \mathbf{x} \in \Omega^c, \\
u(\mathbf{x}) &= g(\mathbf{x}), \quad \mathbf{x} \in \partial \Omega_D
\\
\frac{\partial u(\mathbf{x})}{\partial \mathbf{n}}&=q(\mathbf{x}),\quad \mathbf{x} \in \partial \Omega_N
\end{split}
\end{equation}
where $\partial \Omega_D$ and $\partial \Omega_N$ denote complementary portions of the boundary $\partial \Omega$, satisfying
\begin{equation}
\label{eq:boundary}
\partial\Omega_D \cup \partial\Omega_N = \partial\Omega, \quad \partial\Omega_D \cap \partial\Omega_N = \varnothing.
\end{equation}
Here, $\mathbf{n}$ denotes the outward unit normal vector on $\partial\Omega$.

This work considers both interior and exterior BVPs, in which the governing equation $\mathcal{L}
u = 0$ is posed in $\Omega$ and $\Omega^c$, respectively. The differential operator $\mathcal{L}$ is chosen to be either Laplace operator $-\Delta $
or Helmholtz operator $-\Delta-k^2$, where $k$ denotes the wavenumber. To analyze these BVPs, we introduce the fundamental solutions associated with the differential operator.

\subsection{Fundamental solution of differential operator}
Given a point $\mathbf{x} \in \mathbb{R}^d$, the fundamental solution (or free-space Green's function) $\Phi(\mathbf{x},\mathbf{y})$ associated with the differential operator $\mathcal{L}_{\mathbf{y}}$ is defined as the solution of the following PDE: 
\begin{equation} 
\label{eq:green_pde}
\begin{split}
\mathcal{L}_{\mathbf{y}} \Phi(\mathbf{x},\mathbf{y}) = \delta(\mathbf{x}-\mathbf{y}) &, \quad \mathbf{y} \in \mathbb{R}^d. \\
\end{split}
\end{equation}
where $\delta(\cdot)$ denotes the Dirac delta function and the operator $\mathcal{L}_{\mathbf{y}}$ acts on the $\mathbf{y}$ variable.

For many important linear PDEs, the fundamental solutions can be written explicitly  \cite{duffy2015green}. In particular, for the Laplace equation $- \Delta u =0$, the fundamental solution is given by
\begin{align}
    \Phi(\mathbf{x},\mathbf{y})=
    \begin{cases}
    -\dfrac{1}{2\pi}\ln|\mathbf{x-y}|,& d=2,\\[8pt]
\dfrac{1}{d(d-2)\,\alpha(d)}\dfrac{1}{\mathbf{|x-y|}^{\,d-2}},& d\ge3,
\end{cases}
\end{align}
where $\alpha(d)$ is the volume of the unit ball in $\mathbb{R}^d$. For the Helmholtz equation $-\Delta u -k^2u = 0$, the fundamental solution $\Phi_k(\mathbf{x},\mathbf{y})$ is given by
\begin{align}
\Phi_k(\mathbf{x},\mathbf{y})=
  \begin{cases}
\displaystyle \frac{\mathrm{i}}{4} H_0^{(1)}(k\mathbf{|x-y|}), & d=2, \\[10pt]
\displaystyle \frac{e^{\mathrm{i}k\mathbf{|x-y|}}}{4 \pi\mathbf{|x-y|}}, & d=3, \\[10pt]
\displaystyle \frac{ \mathrm{i} \left(\frac{k}{2 \pi\mathbf{|x-y|}}\right) ^{\frac{n}{2}-1} H_{\frac{n}{2}-1}^{(1)}(k\mathbf{|x-y|})}{4}, & d \geq 4.
\end{cases}
\end{align}
where $H_n^{(1)}$ is the Hankel function \cite{evans2022partial} of the first kind of order $n$. The fundamental solution $\Phi$ is the kernel used to construct the layer potentials that form the basis of the boundary integral formulation.

\subsection{Potential Theory}

Let $\Omega$ be a bounded domain with smooth boundary $\partial\Omega$.  The Green's second identity \cite{evans2022partial} states that for functions $u,v\in C^2(\Omega)\cap C^{1}(\overline{\Omega})$, 
\begin{equation}
\label{eq:eq5}
    \int_{\Omega}(u\mathcal{L} v-v\mathcal{L} u)\mathrm{d}\mathbf{y}=\int_{\partial\Omega}\left(u\frac{\partial v}{\partial\mathbf{n}}-v\frac{\partial u}{\partial\mathbf{n}}\right) \mathrm{d}S(\mathbf{y}).
\end{equation}
where $\mathbf{n}$ denotes the outer normal of $\partial\Omega$ at point $\mathbf{y}$ and $\mathrm{d}S(\mathbf{y})$ denotes the area element at the boundary point
$\mathbf{y}$. Let $u$ be a solution to the homogeneous PDE $\mathcal{L} u=0$ in $\Omega$. For a fixed $\mathbf{x} \in \Omega$, the fundamental solution $\Phi(\mathbf{x},\cdot)$ has a singularity at $\mathbf{x}$, to handle this singularity, we  apply \eqref{eq:eq5} in the domain \(\Omega_\varepsilon = \Omega \setminus B_\varepsilon(\mathbf{x})\), where \(B_\varepsilon(\mathbf{x})\) is a small ball of radius \(\varepsilon\) centered at \(\mathbf{x}\). Substituting the property of the fundamental solution given in \eqref{eq:green_pde} and then letting \(\varepsilon \to 0\), 
we obtain the following representation formula for \(u\) at \(\mathbf{x}\):
\begin{equation}
\label{eq:green_integral}
u(\mathbf{x})=\int_{\partial\Omega}\left[\frac{\partial\Phi(\mathbf{x},\mathbf{y})}{\partial\mathbf{n}}u(\mathbf{y})-\Phi(\mathbf{x},\mathbf{y})\frac{\partial u(\mathbf{y})}{\partial\mathbf{n}}\right] \mathrm{d}S(\mathbf{y}).
\end{equation}
Equation \eqref{eq:green_integral} shows that the solution $u(\mathbf{x})$
at any interior point can be expressed entirely in terms of its boundary values $\left.u\right|_{\partial \Omega}$ and its boundary normal derivative $\frac{\partial u}{\partial \mathbf{n}} \bigg|_{\partial \Omega}$. The terms in the integrand motivate the definitions of the double-layer and single-layer potentials, respectively.

Given the differential operator $\mathcal{L}$ and the corresponding fundamental solution $\Phi(\mathbf{x},\mathbf{y})$, for any continuous function $\hat{h}$ defined on $ \partial \Omega$, the single layer potential is defined as 
\begin{equation} 
\label{eq:single_layer_potential}
\begin{split}
\mathcal{S}[\hat{h}](\mathbf{x}) =- \int_{\partial \Omega} \Phi(\mathbf{x},\mathbf{y}) \hat{h}(\mathbf{y}) \mathrm{d}S(\mathbf{y}).
\end{split}
\end{equation}
 and the double layer potential is defined as, 
\begin{equation} 
\label{eq:double_layer_potential}
\begin{split}
\mathcal{D}[\hat{h}](\mathbf{x}) = -\int_{\partial \Omega} \frac{\partial\Phi(\mathbf{x},\mathbf{y})}{\partial \mathbf{n}} \hat{h}(\mathbf{y}) \mathrm{d}S(\mathbf{y}).
\end{split}
\end{equation}

Based on the layer potential theory of the Poisson equation and the Helmholtz equation, we have the following properties \cite{binet_Lin2023}\cite{helms2009potential}:
\begin{itemize}
    \item The single and double layer potentials satisfy \eqref{eq:harmonic_pde}, specifically, for $\mathbf{x}\notin\partial\Omega$, we have 
    \begin{equation} 
    \label{eq:property_1}
    \begin{split}
    \mathcal{L} \mathcal{S}[\hat{h}](\mathbf{x}) &= 0, \quad  \\
    \mathcal{L} \mathcal{D}[\hat{h}](\mathbf{x}) &= 0, \quad 
    \end{split}
    \end{equation}
    
    \item For $\mathbf{x}_0 \in \partial \Omega$, and the boundary near $\mathbf{x}_0$ is smooth, we have
    \begin{equation} 
    \label{eq:property_2}
    \begin{split}
    \lim_{\mathbf{x} \to\mathbf{x}_0} \mathcal{S}[\hat{h}](\mathbf{x}) &= \mathcal{S}[\hat{h}](\mathbf{x}_0), \\
    \lim_{\mathbf{x} \to\mathbf{x}_{0}^\pm}  \mathcal{D}[\hat{h}](\mathbf{x}) &= \mathcal{D}[\hat{h}](\mathbf{x}_0) \mp  \frac{1}{2} \hat{h}(\mathbf{x}_0),
    \end{split}
    \end{equation}
    where $\mathbf{x}\to\mathbf{x_0}^{-}$ and $\mathbf{x}\to\mathbf{x_0}^{+}$ mean convergence in $\Omega$ and $\Omega^{c}$, respectively.
\end{itemize}

\subsection{Boundary Integral Formulation}

Consider the interior problem of $\eqref{eq:harmonic_pde}$, where $\mathcal{L}u(\mathbf{x})=0$ is defined for $\mathbf{x}\in\Omega$. The boundary conditions could be prescribed as follows:
\begin{itemize}
    \item On the dirichlet boundary $\partial\Omega_D$, the potential value is given by $u(\mathbf{x})=g(\mathbf{x})$, while its normal derivative, denoted by $q_D(\mathbf{x}):=\frac{\partial u(\mathbf{x})}{\partial\mathbf{n}}$, is unknown.
    \item On the Neumann boundary $\partial\Omega_N$, the normal derivative is given by $\frac{\partial u(\mathbf{x})}{\partial\mathbf{n}}=q(\mathbf{x})$, while the potential value, denoted by $g_N(\mathbf{x}):=u(\mathbf{x})$, is unknown.
\end{itemize}
Substituting the known boundary conditions and the unknown quantities into the equation  \eqref{eq:green_integral}, the solution $u(\mathbf{x})$ 
can be expressed as:
\begin{equation} 
\label{eq:interior_exact}
u(\mathbf{x}) = \int_{\partial\Omega_D} \frac{\partial\Phi(\mathbf{x}, \mathbf{y})}{\partial\mathbf{n}} g(\mathbf{y}) \, \mathrm{d}S(\mathbf{y}) - \int_{\partial\Omega_D} \Phi(\mathbf{x}, \mathbf{y}) q_D(\mathbf{y}) \, \mathrm{d}S(\mathbf{y})+\int_{\partial\Omega_N} \frac{\partial\Phi(\mathbf{x}, \mathbf{y})}{\partial\mathbf{n}} g_N(\mathbf{y}) \, \mathrm{d}S(\mathbf{y})- \int_{\partial\Omega_N} \Phi(\mathbf{x}, \mathbf{y}) q(\mathbf{y}) \, \mathrm{d}S(\mathbf{y}) 
\end{equation}
where $\mathbf{x} \in \Omega$.
 To solve for unknown functions $q_D$ and $g_N$, we let the field point $\mathbf{x}$ approach a point on the boundary from within the domain. According to \eqref{eq:property_2}, we obtain a system of coupled boundary integral equations: 
\begin{itemize}
    \item For a boundary point $\mathbf{x}_0\in\partial\Omega_D$, the equation for the unknown $q_D$, $g_N$ is derived as:
\begin{equation}
\label{eq:bie-dirichlet-boundary}
\begin{aligned}
\frac{1}{2}g(\mathbf{x}_{0}) = & \int_{\partial \Omega_{D}} \frac{\partial \Phi\left(\mathbf{x}_{0}, \mathbf{y}\right)}{\partial \mathbf{n}} g(\mathbf{y}) \mathrm{d} S(\mathbf{y})
- \int_{\partial \Omega_{D}} \Phi\left(\mathbf{x}_{0}, \mathbf{y}\right) q_{D}(\mathbf{y}) \mathrm{d} S(\mathbf{y}) \\
& + \int_{\partial \Omega_{N}} \frac{\partial \Phi\left(\mathbf{x}_{0}, \mathbf{y}\right)}{\partial \mathbf{n}} g_{N}(\mathbf{y}) \mathrm{d} S(\mathbf{y})
- \int_{\partial \Omega_{N}} \Phi\left(\mathbf{x}_{0}, \mathbf{y}\right) q(\mathbf{y}) \mathrm{d} S(\mathbf{y})
\end{aligned}
\end{equation}

    \item For a boundary point $\mathbf{x}_0\in\partial\Omega_N$, the equation for the unknown $q_D$, $g_N$ is derived as:
\begin{equation}
\label{eq:bie-neumann-boundary}
\begin{aligned}
\frac{1}{2}g_N(\mathbf{x}_{0}) = & \int_{\partial \Omega_{D}} \frac{\partial \Phi\left(\mathbf{x}_{0}, \mathbf{y}\right)}{\partial \mathbf{n}} g(\mathbf{y}) \mathrm{d} S(\mathbf{y})
- \int_{\partial \Omega_{D}} \Phi\left(\mathbf{x}_{0}, \mathbf{y}\right) q_{D}(\mathbf{y}) \mathrm{d} S(\mathbf{y}) \\
& + \int_{\partial \Omega_{N}} \frac{\partial \Phi\left(\mathbf{x}_{0}, \mathbf{y}\right)}{\partial \mathbf{n}} g_{N}(\mathbf{y}) \mathrm{d} S(\mathbf{y})
- \int_{\partial \Omega_{N}} \Phi\left(\mathbf{x}_{0}, \mathbf{y}\right) q(\mathbf{y}) \mathrm{d} S(\mathbf{y})
\end{aligned}
\end{equation}
\end{itemize}
Thus, we have successfully transformed the partial differential equation problem defined in the interior domain into a boundary integral equation problem that is solved solely on the boundary. These two equations constitute a coupled system for solving the unknowns $q_D$ and $g_N$. Once this system is solved, the complete boundary information is obtained and the solution can be computed at any point in the domain using \eqref{eq:interior_exact}. For the exterior problem of \eqref{eq:harmonic_pde}, the derivation is similar.

\section{Methods}
\label{sec:BNM}

This section is structured in two parts. We begin by outlining the standard Boundary Neuron Method (BNM) and then introduce our principal contribution: the BNM with random feature. 
\subsection{Boundary neuron method}
\label{sec:bnm-shallow-network}
The shallow neural network constructs an approximate solution $u_M$ in the following form:
\begin{equation}
\label{eq:u_out}
u_{M}(\mathbf{x})=\sum_{j=1}^{M}\beta_{j}\phi_{j}(\mathbf{x})
\end{equation}
Here, $M$ represents the number of neurons, $\beta_j$ are the expansion coefficients and $\phi_j$ are neurons. The neurons are parameterized by internal weights $\mathbf{\omega}_{j}\in \mathbb{R}^{d}$ and biases $b_j\in\mathbb{R}$ and selecting a smooth activation function $\sigma:\mathbb{R} \to \mathbb{R}$ (e.g.,
tanh or cosine), their expression is:
\begin{equation}
    \phi_{j}(\mathbf{x})=\sigma\left(\left\langle\mathbf{\omega}_{j},\mathbf{x}\right\rangle+b_j\right)
\label{eq:activation}
\end{equation}
where $\left\langle\cdot,\cdot\right\rangle$ denotes the inner product of vectors.

This section presents the boundary neuron method. The core idea is to approximate the unknown function in the BIE using a shallow neural network:
\begin{equation}
\label{eq:q_D_and_g_N}    q_D(\mathbf{y})=\frac{\partial u_{M}}{\partial \mathbf{n}}(\mathbf{y}),\quad g_N(\mathbf{z})=u_{M}(\mathbf{z}),\quad
\text{where}\quad \mathbf{y}\in\partial\Omega_D,\quad \mathbf{z}\in\partial\Omega_N 
\end{equation}
It's important to note that when $\mathbf{y}\in\partial\Omega_D$, $u_M(\mathbf{y})$ and $g(\mathbf{y})$ may not be equal; and when $\mathbf{y}\in\partial\Omega_N$, $\frac{\partial u_M(\mathbf{y})}{\partial \mathbf{n}}$ and $q(\mathbf{y})$ may also not match. For further reference, one can refer to BINN \cite{binn_Sun2023}. We propose a novel boundary neuron method within a  random feature framework
\cite{chen2022bridging} for estimating the model parameters.  

\subsection{Boundary neuron method with random feature}

As mentioned above, the BNM ultimately requires the solving of the boundary integral equation \eqref{eq:interior_exact}. This section explains how to use random feature to efficiently solve this equation.

This method employs a shallow neural network that contains only a single hidden layer of $M$ neurons. The weights and biases $\left\{ \left( \mathbf{\omega}_{j}, b_{j} \right) \right\}_{j=1}^{M}$ are randomly sampled and fixed, defining a set of basis functions $\left\{ \phi_{j} \right\}_{j=1}^{M}$. Consequently, the final output of the network is formulated as a linear combination of these pre-defined basis functions as \eqref{eq:u_out}. Next, it is necessary to substitute  \eqref{eq:q_D_and_g_N} into the boundary integral equations \eqref{eq:bie-dirichlet-boundary} and
\eqref{eq:bie-neumann-boundary} to construct the linear system $\mathbf{A}\beta=\mathbf{b}$. Specifically, substituting these expressions into \eqref{eq:bie-dirichlet-boundary} for a Dirichlet boundary point $\mathbf{x_0}\in\partial\Omega_D$ yields 
\begin{equation*}
\begin{split}
    \frac{1}{2} g(\mathbf{x}_0)
    &-\int_{\partial\Omega_D} \frac{\partial\Phi(\mathbf{x}_0,\mathbf{y})}{\partial\mathbf{n}} g(\mathbf{y}) \, \mathrm{d}S(\mathbf{y})
    +\int_{\partial\Omega_N} \Phi(\mathbf{x}_0,\mathbf{y}) q(\mathbf{y}) \, \mathrm{d}S(\mathbf{y}) \\
    &= \sum_{j=1}^M \Bigg(
        -\int_{\partial\Omega_D} \Phi(\mathbf{x}_0,\mathbf{y}) \frac{\partial\phi_j}{\partial\mathbf{n}}(\mathbf{y}) \, \mathrm{d}S(\mathbf{y})
        +\int_{\partial\Omega_N} \frac{\partial\Phi(\mathbf{x}_0,\mathbf{y})}{\partial\mathbf{n}} \phi_j(\mathbf{y}) \, \mathrm{d}S(\mathbf{y})
        \Bigg) \beta_j
\end{split}
\end{equation*}
and into $\eqref{eq:bie-neumann-boundary}$ for a Neumann boundary point $\mathbf{x}_0\in\partial\Omega_N$ yields
\begin{equation*}
\begin{aligned}
\int_{\partial\Omega_D} \frac{\partial\Phi(\mathbf{x}_0,\mathbf{y})}{\partial\mathbf{n}} g(\mathbf{y}) \, \mathrm{d}S(\mathbf{y}) 
&-\int_{\partial\Omega_N} \Phi(\mathbf{x}_0,\mathbf{y}) q(\mathbf{y}) \, \mathrm{d}S(\mathbf{y}) \\
&= \sum_{j=1}^M \Bigg( \int_{\partial\Omega_D} \Phi(\mathbf{x}_0,\mathbf{y})  \frac{\partial\phi_j}{\partial\mathbf{n}}(\mathbf{y}) \, \mathrm{d}S(\mathbf{y}) 
 -\int_{\partial\Omega_N} \frac{\partial\Phi(\mathbf{x}_0,\mathbf{y})}{\partial\mathbf{n}}  \phi_j(\mathbf{y}) \, \mathrm{d}S(\mathbf{y})
+\frac12 \phi_j(\mathbf{x}_0) \Bigg)\beta_j
\end{aligned}
\end{equation*}
We enforce these equations at all boundary collocation points ${\{\mathbf{x}^{i}\}}_{i=1}^{N_s}$, where $N_s$ is the total number of collocation points on both the dirichlet and neumann boundaries.

\paragraph{Construction of Matrix \(\mathbf{A}\)} Each row of the matrix corresponds to a collocation point \(\mathbf{x}^i\) and a boundary integral equation. The element \(\mathbf{A}[i,j]\) represents the contribution of the \(j\)-th basis function \(\phi_j\)  at point \(\mathbf{x}^i\). Depending on the form of the equation, the contribution may come from either the basis function itself or its normal derivative.
\begin{itemize}
    \item If \(\mathbf{x}^i \in \partial \Omega_D\), we use the equation corresponding to the dirichlet boundary. In this case, the contribution of the basis function takes the form:
    \begin{equation}
    \label{eq:A_ij_d}
    \mathbf{A}[i, j] = -\int_{\partial \Omega_D} \Phi(\mathbf{x}^i, \mathbf{y}) \frac{\partial \phi_j}{\partial \mathbf{n}}(\mathbf{y}) \, \mathrm{d}S(\mathbf{y}) + \int_{\partial \Omega_N} \frac{\partial \Phi(\mathbf{x}^i, \mathbf{y})}{\partial \mathbf{n}} \phi_j(\mathbf{y}) \, \mathrm{d}S(\mathbf{y})
    \end{equation}

    \item If \(\mathbf{x}^i \in \partial \Omega_N\), we use the equation corresponding to the neumann boundary. In this case, the contribution takes the form:
    \begin{equation}
    \label{eq:A_ij_n}
    \mathbf{A}[i, j] = -\int_{\partial \Omega_D} \Phi(\mathbf{x}^i, \mathbf{y}) \frac{\partial \phi_j}{\partial \mathbf{n}}(\mathbf{y}) \, \mathrm{d}S(\mathbf{y}) + \int_{\partial \Omega_N} \frac{\partial \Phi(\mathbf{x}^i, \mathbf{y})}{\partial \mathbf{n}} \phi_j(\mathbf{y}) \, \mathrm{d}S(\mathbf{y})-\frac{1}{2}\phi_j(\mathbf{x}^i)
    \end{equation}
    \end{itemize}
  The final dimension of the matrix is \(N_s \times M\).

\paragraph{Construction of vector $\mathbf{b}$} Each component $\mathbf{b}[i]$ of the vector corresponds to the contribution of the known boundary conditions at the collocation point $\mathbf{x}^{i}$.

\begin{itemize}
  \item If $\mathbf{x}^{i}\in\partial\Omega_{D}$:
  \begin{equation}
  \label{eq:b_i_d}
  \mathbf{b}[i]=\frac{1}{2}g(\mathbf{x}^{i})-\int_{\partial\Omega_{D}}\frac{\partial\Phi(\mathbf{x}^{i},\mathbf{y})}{\partial \mathbf{n}}g(\mathbf{y})\mathrm{d}S(\mathbf{y})+\int_{\partial\Omega_{N}}\Phi(\mathbf{x}^{i},\mathbf{y})q(\mathbf{y})\mathrm{d}S(\mathbf{y})
  \end{equation}

  \item If $\mathbf{x}^{i}\in\partial\Omega_{N}$:
  \begin{equation}
  \label{eq:b_i_n}
  \mathbf{b}[i]=-\int_{\partial\Omega_{D}}\frac{\partial\Phi(\mathbf{x}^{i},\mathbf{y})}{\partial \mathbf{n}}g(\mathbf{y})\mathrm{d}S(\mathbf{y})+\int_{\partial\Omega_{N}}\Phi(\mathbf{x}^{i},\mathbf{y})q(\mathbf{y})\mathrm{d}S(\mathbf{y})
  \end{equation}
\end{itemize}
The integrals above are evaluated using numerical method. We can define a unified numerical quadrature operator $\mathcal{G}$ to describe the numerical approximation of the above integrals. For any boundary integral, this operator can be expressed as:

\begin{equation}
    \label{eq:gauss_quad}
    \int_{\partial\Omega}f(\mathbf{y})\mathrm{d}S(\mathbf{y})\approx\mathcal{G}[f]= \sum_{n_q=1}^{N_q}\omega_{n_q}f(\mathbf{y}_{n_q})
\end{equation}
Here, $\partial\Omega$ represents the boundary of the integration domain, $f(\mathbf{y})$ is the integrand function, and $N_q$ denotes the number of quadrature points. $\mathbf{y}_{n_q}$ represents the $n_q$-th quadrature point and $\omega_{n_q}$ is the corresponding quadrature weight.
In this work, Gaussian quadrature is adopted for its high accuracy in evaluating such integrals.

Through the above process, the problem of solving the system of boundary integral equations is transformed into a linear least squares problem $\mathbf{A}\beta = \mathbf{b}$ for solving coefficients $\beta$. After obtaining the optimal coefficient vector $\mathbf{\beta^{*}}$ by solving the linear system
$\mathbf{A\beta=b}$, the final predicted solution $\tilde{u}$ at any point $\mathbf{x}$ within the domain $\Omega$ is constructed by superimposing the contributions from the neural network approximation and the boundary condition terms. 
Specifically, we substitute $\beta^{*}$ into \eqref{eq:u_out} to obtain the trained neural network output $u_{M}^{*}(\mathbf{y})$. Then, by replacing $g_N$ and $q_D$ in equation \eqref{eq:interior_exact} with $u_{M}^{*}(\mathbf{y})$ and its normal derivative respectively, we obtain the final predicted solution
\begin{equation}
\label{eq:pred_omega_solution}
    \tilde{u}(\mathbf{x})=\int_{\partial\Omega_D} \frac{\partial\Phi\mathbf{(x, y)}}{\partial\mathbf{n}} g(\mathbf{y}) \, \mathrm{d}S(\mathbf{y}) - \int_{\partial\Omega_D} \Phi\mathbf{(x, y)} \frac{\partial u_{M}^{*}(\mathbf{y})}{\partial \mathbf{n}} \, \mathrm{d}S(\mathbf{y})+\int_{\partial\Omega_N} \frac{\partial\Phi\mathbf{(x, y)}}{\partial\mathbf{n}} u_{M}^{*}(\mathbf{y}) \, \mathrm{d}S(\mathbf{y})- \int_{\partial\Omega_N} \Phi\mathbf{(x, y)} q(\mathbf{y}) \, \mathrm{d}S(\mathbf{y})
\end{equation}

In summary, we conclude the algorithm of boundary neuron method with random feature in Algorithm \ref{algo:bnm-rf-scheme}.

\begin{algorithm}[ht]
\caption{Numerical scheme of boundary neuron method with random feature}
\label{algo:bnm-rf-scheme}

\KwIn{Boundary collocation point set ${\{\mathbf{x}^i\}}_{i=1}^{N_s}$, random weights and biases $\left\{ \left( \mathbf{\omega}_{j}, b_{j} \right) \right\}_{j=1}^{M}$, Quadrature weights and points $\left\{ \left( {\omega}_{n_q}, \mathbf{y}_{n_q} \right) \right\}_{n_q=1}^{N_q}$, fundamental function and its outward normal derivative $\Phi(\mathbf{x}, \mathbf{y})$ and $\frac{\partial \Phi(\mathbf{x}, \mathbf{y})}{\partial \mathbf{n}}$ \;}
\KwOut{Optimal coefficient vector  $\beta^*$, Neural network output $u^{*}_M(\mathbf{y})$, Final predicted solution $\tilde{u}(\mathbf{x})$\;}
Construct matrix $\mathbf{A} \in \mathbb{R}^{N_s \times M}$ through \eqref{eq:A_ij_d} and \eqref{eq:A_ij_n} by numerical quadrature \eqref{eq:gauss_quad} \;
Construct vector $\mathbf{b} \in \mathbb{R}^{N_s}$ through \eqref{eq:b_i_d} and \eqref{eq:b_i_n} by numerical quadrature \eqref{eq:gauss_quad} \;
Solve the linear least-squares problem $\mathbf{A}\beta=\mathbf{b}$ to obtain $\beta^*$ \;
Substitute $\beta^{*}$ into \eqref{eq:u_out} to obtain the trained neural network output $u_{M}^{*}(\mathbf{y})$\;
Replace $g_N$ and $q_D$ in equation \eqref{eq:interior_exact} with $u_{M}^{*}(\mathbf{y})$ and $\frac{\partial u_{M}^{*}(\mathbf{y})}{\partial \mathbf{n}}$ to obtain the final predicted solution $\tilde{u}(\mathbf{x})$
\end{algorithm}

\section{Convergence analysis on BNM-RF}
\label{sec:analysis}
In this section, we show the convergence analysis of our method. Our convergence analysis relies on the standard convergence analysis of kernel method in \cite{batlle2025error} and the kernel approximation by using random feature. Following the framework of \cite{liao2025solving}, where the convergence rate of the random feature method was analyzed for strong-form equations, we employ the similar approach to establish the convergence rate for the boundary neuron method with random feature.

\subsection{Elliptic Boundary Value Problem}
We have described the problem in Section \ref{sec:problem-description}. In this section, to simplify the discussion, we consider elliptic partial differential equations with Neumann boundary condition.  
Using the boundary integral method, the boundary value problem above can be transformed into a boundary integral equation. 

\begin{equation}
\label{eq:BIER}
    \frac{1}{2}u(\mathbf{x})-\int_{\partial\Omega} \frac{\partial
    \Phi(\mathbf{x},\mathbf{y})}{\partial \mathbf{n}(\mathbf{y})} \, u(\mathbf{y}) \, dS(\mathbf{y})=-\int_{\partial\Omega} { \Phi(\mathbf{x},\mathbf{y})} \, g_N(\mathbf{y}) \, dS(\mathbf{y}). \quad \mathbf{x}\in\partial\Omega
\end{equation}
Consider the second-kind Fredholm integral operator:
$$
\mathcal{T}(u) = \left( \frac{1}{2}\mathcal{I} - \mathcal{K} \right)(u)
$$
where $\mathcal{I}$ is the identity operator and $\mathcal{K}$ is a integral operator.
$$
\mathcal{K}(u)(\mathbf{x}) = \int_{\partial\Omega} \frac{\partial \Phi(\mathbf{x},\mathbf{y})}{\partial \mathbf{n}(\mathbf{y})} \, u(\mathbf{y}) \, dS(\mathbf{y}).\quad \mathbf{x}\in\partial\Omega
$$
The original problem \eqref{eq:BIER} can be transformed to
\begin{equation}
\mathcal{T}(u)(\mathbf{x}) = f(\mathbf{x}), \quad \mathbf{x} \in \partial\Omega
\end{equation}
with $f(\mathbf{x}) := -\int_{\partial\Omega} \Phi(\mathbf{x}, \mathbf{y}) g_{N}(\mathbf{y}) \, dS(\mathbf{y})$.\\
To represent this equation compactly, we define the boundary residual operator 
\begin{equation}
\label{eq:second-fredholm}
    \mathcal{B}[u](\mathbf{x})=\mathcal{T}(u)(\mathbf{x})-f(\mathbf{x})=0,\quad \mathbf{x}\in\partial\Omega
\end{equation}

\subsubsection{Error analysis of kernel-based method}
Recall the kernel-based method for solving PDEs in \cite{chen2021solving}\cite{batlle2025error}, a reproducing kernel Hilbert space (RKHS) $\mathcal{H}$  is chosen and we aim to solve the following
\begin{equation}
\label{eq:RKHS-mini}
\begin{aligned}
& \underset{u \in \mathcal{H}}{\text{minimize }} 
  && \|u\|_{\mathcal{H}} \\
& \text{s.t.} 
& && \mathcal{B}[u](\mathbf{x}_{j}) = 0, \quad 
     \text{for } j =  1, \dots, N_s
\end{aligned}
\end{equation}
We first state the main assumptions on the domain $\Omega$ and its boundary $\partial\Omega$, the operator
$\mathcal{T}$, and the reproducing kernel Hilbert space $\mathcal{H}$.
\begin{assumption}[Regularity of the domain and its boundary] $\Omega\in \mathbb{R}^n$ \label{ass:regularity}
with $d>1$ is a compact set and $\partial\Omega$ is a piecewise smooth connected Riemannian manifold of dimension $d-1$ endowed with a geodesic distance $\rho_{\Gamma}$.
\end{assumption}

\begin{assumption}[Invertibility of the boundary integral operator] 
\label{ass:Invertibility of operator}
We assume that the wavenumber \( k > 0 \) is such that \( k^2 \) is not an interior Dirichlet eigenvalue for \( -\Delta \) in \(\Omega \). Under this condition and the smoothness assumption on \(\partial\Omega\) (Assumption \ref{ass:regularity}), the boundary integral operator  
$\mathcal{T}$ is bijective with a bounded inverse (Theorem 3.8.7,     Remark 3.4.3,
\cite{sauter2010boundary}). Moreover, the following stability estimates hold: there exist constants $C > 0$, $\gamma > 0$
and  $s$
such that for any $r > 0$,
\begin{align*}
&\text{(i) } \|u_1 - u_2\|_{H^{\frac{1}{2}}(\partial\Omega)} \leq C \|\mathcal{T}(u_1) - \mathcal{T}(u_2)\|_{H^{\frac{1}{2}}(\partial\Omega)}, \quad \forall u_1, u_2 \in B_r(H^{\frac{1}{2}}(\partial\Omega)), \\
&\text{(ii) } \|\mathcal{T}(u_1) - \mathcal{T}(u_2)\|_{H^{\frac{1}{2}+\gamma}(\partial\Omega)} \leq C \|u_1 - u_2\|_{H^{\frac{1}{2}+\gamma}(\partial\Omega)}, \quad \forall u_1, u_2 \in B_r(H^{\frac{1}{2}+\gamma}(\partial\Omega)),
\end{align*}
where $B_r(X)$
denotes the closed ball of radius $r$  in the Banach space X .
\end{assumption}

\begin{assumption} \label{ass:embedded}
The RKHS $\mathcal{H}$ is continuously embedded in $H^{s}(\partial\Omega)$.
\end{assumption}
Assumption \ref{ass:regularity} is a standard requirement on the regularity of the solution domain and its boundary. It serves as the theoretical foundation for analyzing boundary integral equations. Assumption \ref{ass:Invertibility of operator} states the invertibility and stability of the boundary integral operator $\mathcal{T}$. This ensures the well-posedness of the solution.
Assumption \ref{ass:embedded} connects the norm of the Reproducing Kernel Hilbert Space
$\mathcal{H}$ with the norm of the Sobolev space $H^s(\partial\Omega)$ and  suggests that the reproducing kernel Hilbert space
$\mathcal{H}$ and its kernel function should be chosen to match the regularity of the solution $u$.
We are now ready to present the first theorem which concerns the convergence rate of kernel method.

\begin{theorem}
\label{thm:RKHS-solution-error}
Let $\hat{u}$ be a
minimizer of \eqref{eq:RKHS-mini} with collocation points $X_{\Gamma}$ on the boundary . Define the fill-in distances 
\begin{equation*}
    h_{\Gamma} := \sup_{\mathbf{x}^{\prime} \in \partial\Omega} \inf_{\mathbf{x} \in X_{\Gamma}} \rho_{\Gamma}(\mathbf{x}, \mathbf{x}^{\prime}),
\end{equation*}
and set $h=h_{\Gamma}$.
Then there exists a constant $h_0$ so that if $h<h_0$ then:
\begin{equation}
\|u - \hat{u}\|_{H^{\frac{1}{2}}(\partial\Omega)} \leq C h^\gamma \|u\|_{\mathcal{H}}
\end{equation}
where $C > 0$ is a constant independent of $h$ and $u$.
\end{theorem}
\begin{proof}
  According to Theorem \ref{thm:abstract_banach_spaces}, we only need to verify six assumptions. If all assumptions are satisfied, the conclusion follows. We then choose the spaces $V_1 = V_2 = H^{\frac{1}{2}}(\partial\Omega)$, $V_3 = V_4 = H^{\gamma+\frac{1}{2}}(\partial\Omega)$ where we equip $V_2$ with the norm $\|f\|_{2} := \|f\|_{H^{\frac{1}{2}}(\partial\Omega)}$ and similarly for $V_3$ with the norm $\|f\|_{3}:= \|f\|_{H^{\gamma+\frac{1}{2}}(\partial\Omega)}$, and we take the operator ${\mathcal{T}}: u\mapsto{\mathcal{T}}(u)$. We can verify conditions (A1), (A2), and (A3) by the Assumption \ref{ass:regularity} - Assumption \ref{ass:embedded}. Condition (A6) is also satisfied since
  $\hat{u}$ is a minimizer of \eqref{eq:RKHS-mini}, let $\bar{g} = \mathcal{B}(\hat{u}) - \mathcal{B}(u)$ satisfy $\bar{g}(\mathbf{x}) = 0$ for all $\mathbf{x} \in X_{\Gamma}$ and so $\bar{g}\in H^{\gamma+\frac{1}{2}}(\partial\Omega)$ is zero on the set $X_{\Gamma}$. Then Proposition\ref{prop:A1} implies the existence of a constant $h_0>0$ so that whenever $h_{\Gamma}<h_0$ we have $\|\bar{g}\|_{H^{\frac{1}{2}}(\partial\Omega)} \leq C h_{\Gamma}^{\gamma} \|\bar{g}\|_{H^{\gamma+\frac{1}{2}}(\partial\Omega)}.$ This verifies (A4) with $\varepsilon \equiv Ch_\Gamma^{\gamma}$.
\end{proof}

\subsubsection{Random feature-RKHS error}
With the kernel minimizer $\hat{u} \in \mathcal{H}$ at hand, our next step is approximating $\hat{u}$ using random feature. We first adopt an alternative representation of preselected RKHS $\mathcal{H}$. 
Specifically, if the reproducing kernel $K$ of $\mathcal{H}$ is shift-invariant, and $K$ is continuous and positive definite, then by Bochner's theorem \cite{bochner2005harmonic}, the Fourier transform of kernel $K$ establishes a spectral measure $\rho(\omega)$, defined explicitly by
$$
\rho(\omega) = \int_{\mathbb{R}^d} \exp \bigl(-\mathrm{i}\langle \omega, \zeta \rangle\bigr) \, K(\zeta) \, d\zeta
$$
so that the kernel can be written in the expectation form
$$
K(\mathbf{x}, \mathbf{x}') = \int_{\mathbb{R}^d} \exp     \!\bigl(\mathrm{i}\langle \omega, \mathbf{x - x'} \rangle\bigr) \,
d\rho(\omega)
          = \mathbb{E}_{\omega \sim \rho}\!\bigl[\,\phi(\mathbf{x}, \omega)^* \, \phi(\mathbf{x'}, \omega)\,\bigr]
$$
where the feature map $\phi$ is taken as the random Fourier feature $\phi(\mathbf{x}, \omega) = \exp(\mathrm{i}\langle \omega, \mathbf{x} \rangle)$, or more commonly the  random cosine feature $\phi(\mathbf{x}, \omega) =  \cos(\langle \omega, \mathbf{x} \rangle + b)$. This leads to an equivalent description of $\mathcal{H}$ as the completion of the function space:
$$
\mathcal{F}(f) := \left\{ f(\mathbf{x}) = \int_{\mathbb{R}^{d}} b(\omega) \phi(\mathbf{x}, \omega) \, d\rho(\omega) : \|f\|_{\rho}^{2} = \mathbb{E}_{\omega}[b(\omega)^{2}] < \infty \right\}
$$
Recall Proposition 4.1 in \cite{rahimi2008uniform}, it is indeed the reproducing kernel Hilbert space $\mathcal{H}$ with associated kernel function $K$. The endowed norms $\|\cdot\|_{\mathcal{H}}$ and $\|\cdot\|_{\rho}$ are equivalent.

Examples of different random features $\phi$, together with their corresponding feature functions, parameter distributions, and kernel functions, can be found in \cite{rahimi2008uniform}.

In the present work, we adopt the Fourier-type random feature and use the cosine activation function $\sigma(\mathbf z)=\cos(\mathbf z)$. The frequencies \(\{\omega_j\}\) are sampled independently from \(\mathcal{N}(0,2\gamma I)\) in \(\mathbb{R}^d\), and the biases $\{b_j\}$ are sampled uniformly from $U[-\pi,\pi]$. This leads to the following shallow random feature network:
$$ u_M(\mathbf{x})=\sum_{j=1}^{M} \beta_{j} \phi_{j}(\mathbf{x})=\sum_{j=1}^{M} \beta_{j} \cos \left(\left\langle\mathbf{\omega}_{j}, \mathbf{x}\right\rangle+b_{j}\right), $$
and the coefficients $\beta = (\beta_{1}, \ldots, \beta_{M})^{\top}$ are determined via optimization.

We consider the under-parameterized
case, where the number of random feature $M$ is less than the number of collocation
points $N_s$. In this case, 
the random feature model is trained by solving the following optimization problem:
\begin{equation}
\min_{\beta \in \mathbb{R}^{N}} \sum_{j=1}^{N_{s}} \left| \mathcal{B}[u_{M}](\mathbf{x}_{j}) \right|^{2}
\label{eq:min_problem}
\end{equation}
In the over-parameterized case where the number of random feature $M$ is greater than the number of collocation points $N_s$,  the corresponding training problem becomes \cite{tikhonov1977solutions}:  
\begin{equation}
\label{eq:c-mini}
\begin{aligned}
& \underset{\beta\in \mathbb{R}^N}{\text{minimize }} 
  && \|\beta\|_{2}^2 \\
& \text{s.t.} 
& && \mathcal{B}[u_M](\mathbf{x}_{j}) = 0, \quad 
     \text{for } j =  1, \dots, N_s
\end{aligned}
\end{equation}
This study focuses on the under-parameterized case. The over-parameterized problem will not be discussed here.

The following theorem provides the theoretical guarantee for how well this finite random feature model can approximate a function in $\mathcal{F}$:
\begin{theorem}[Theorem 2,\cite{liao2025solving}]\label{thm:random-feature-approx}
Let $f$ be a function from $\mathcal{F}(f)$. Suppose that the random feature map $\phi$ satisfies $|\phi(\mathbf{x}, \omega)| \leq 1$ for all $\mathbf{x} \in X$ and $\omega \in \mathbb{R}^{d}$. Then for any $\delta \in (0,1)$, there exist coefficients $\beta_{1}^{\sharp}, \ldots, \beta_{M}^{\sharp}$ so that the function
\begin{equation}\label{eq:rf-model}
f^{\sharp}(\mathbf{x}) =  \frac{1}{M}\sum_{j=1}^{M} \beta_{j}^{\sharp} \phi(\mathbf{x}, \omega_{j})
\end{equation}
satisfies
\begin{equation}\label{eq:error-bound}
|f(\mathbf{x}) - f^{\sharp}(\mathbf{x})| \leq \frac{12 \|f\|_{\rho} \log(2/\delta)}{\sqrt{M}}
\end{equation}
with probability at least $1 - \delta$ over $\omega_{1}, \ldots, \omega_{N}$ drawn i.i.d from $\rho(\omega)$.
\end{theorem}

\subsubsection{The final error}
According to Theorem \ref{thm:RKHS-solution-error} and \ref{thm:random-feature-approx}, we can now derive an upper bound for the error on the boundary.
\begin{theorem}
Suppose that the conditions in theorem \ref{thm:RKHS-solution-error} and \ref{thm:random-feature-approx} hold. Then for any $\delta \in (0, 1)$ there exists $\beta_1^m, \ldots, \beta_M^m$ such that the function
$$
u_M(\mathbf{x}) = \sum_{j=1}^{M} \beta_j^m \phi(\mathbf{x}, \omega_j)
$$
satisfies 
$$
\left\|u-u_M\right\|_{L^{2}(\partial\Omega)}\leq Ch^\gamma\|u\|_{\mathcal{H}}+\frac{12\|u\|_{\mathcal{H}}\log (2/\delta)\sqrt{\operatorname{vol}(\partial\Omega)}}{\sqrt{M}}
$$
with probability at least $1-\delta$ over $\omega_1, \ldots, \omega_N$ drawn independently and identically distributed (i.i.d) from $\rho(\omega)$. 
\end{theorem}
\begin{proof}
Using the triangle inequality, we decompose the error as 
\begin{equation*}
    \left\|u-u_M\right\|_{L^{2}(\partial\Omega)}\leq\|u-\hat{u}\|_{L^{2}(\partial\Omega)}+\left\|\hat{u}-u_M\right\|_{L^{2}(\partial\Omega)}
\end{equation*}
 We apply theorem \ref{thm:RKHS-solution-error} to bound  $\|u-\hat{u}\|_{L^{2}(\partial\Omega)}$. The second term is bounded as 
\begin{equation}
    \|\hat{u}-u_M\|_{L^{2}(\partial\Omega)} = \sqrt{\int_{\partial\Omega}|\hat{u}(\mathbf{x})-u_M(\mathbf{x})|^{2} d S(\mathbf{x})} 
\leq \frac{12\|\hat{u}\|_{\mathcal{H}}\log(2/\delta)\sqrt{\operatorname{vol}(\partial\Omega)}}{\sqrt{M}} 
\leq \frac{12\|u\|_{\mathcal{H}}\log(2/\delta)\sqrt{\operatorname{vol}(\partial\Omega)}}{\sqrt{M}}
\end{equation}
\end{proof}
Based on the estimate of the error on the boundary, we now extend the analysis to derive an error bound within the entire domain $\Omega$. This is achieved through the representation formula:
\begin{equation}
    u(\mathbf{x})-\tilde{u}(\mathbf{x})=\int_{\partial\Omega} \frac{\partial G(\mathbf{x},\mathbf{y})}{\partial \mathbf{n}(\mathbf{y})} \, (u(\mathbf{y})-u_M(\mathbf{y})) \, dS(\mathbf{y}),\quad \mathbf{x}\in\Omega
\end{equation}
which expresses the interior error entirely in terms of the error on the boundary. For a Lipschitz boundary $\partial\Omega$, the double-layer potential operator $\mathcal{D}$ is continuous from $L^{2}(\partial\Omega)$ to $H^{\frac{1}{2}}(\Omega)$ (Theorem 1, \cite{costabel1988boundary}). Consequently, there exists a constant $\hat{C}>0$ depending on the wavenumber k and the domain $\Omega$, such that
$$
\|u-\tilde{u}\|_{L^{2}(\Omega)} \leq\|u-\tilde{u}\|_{H^{\frac{1}{2}}(\Omega)}\leq \hat{C} \|u-u_M\|_{L^{2}(\partial\Omega)}.
$$
This inequality confirms the error analysis, as it establishes the convergence result from the boundary to the entire solution domain.

\section{Numerical results}
\label{sec:result}
In this section, we present numerical results for the BNM-RF. In the first experiment, we use the cosine activation function to validate the theoretical convergence rate presented in Section \ref{sec:analysis}. The subsequent experiments  compare the performance of BNM-RF against BINN \cite{binn_Sun2023} and BEM (implemented with constant element) \cite{hall1994boundary}, in these comparative studies, the tanh activation function is adopted. We focus on two types of problems: potential problems governed by the Poisson equation and acoustic wave problems governed by the Helmholtz equation. The accuracy of a trained model is measured by the relative  error $L^{2}$-norm $\frac{\| u - u^{*} \|_{2}}{\| u^{*} \|_{2}}$, where $u$ is the numerical solution, and $u^{*}$ is the reference solution.
 
 \subsection{Interior problem for the Helmholtz equation}
 This experiment is designed to verify the theoretical convergence rate analysis from Section \ref{sec:analysis}, which applies to the interior Helmholtz equation with Neumann boundary condition.
 The specific equations are as follows:
\begin{equation} 
        \label{eq:interior_helmholtz}
        \begin{split}
            - \Delta u(\mathbf{x}) - k^2 u(\mathbf{x}) = 0 &, \quad \mathbf{x} \in \Omega, \\
            \frac{\partial u(\mathbf{x})}{\partial \mathbf{n}} = q(\mathbf{x})&, \quad \mathbf{x} \in \partial \Omega,
        \end{split}
    \end{equation}
where $\Omega=[0,1]\times[0,1]$ and $\mathbf{x} = [x_1, x_2]$,
The analytical solution to this problem is given by the following:
\begin{equation}
\label{eq:helmholt_interior_1}
u(\mathbf{x})=\exp({{\mathrm{i}(k_1x_1+k_2x_2)}}),\quad (k_1,k_2)=(k\cos(\frac{\pi}{7}),k \sin(\frac{\pi}{7})),
\end{equation} 
In the experiments, the square computational domain is considered, whose boundary \(\partial\Omega\) is a Lipschitz boundary. Following the theoretical framework established in Section \ref{sec:analysis}, the regularity index is set to \(\gamma = 0.5\) \cite{costabel1988boundary}\cite{Grisvard1985EllipticPI}, and accordingly the variance parameter is taken as \(2\gamma = 1\). To reduce the effects of random parameter initialization in neural networks, each set of parameter configurations was independently repeated four times, and the final error was averaged over the four results.

We first examined the performance of the BNM-RF at wavenumber $k$=9. The specific experimental parameters were set as follows: the number of hidden neurons $M$=40, and collocation points $N_s$=60 were uniformly selected on the boundary. As shown in Figure \ref{fig:neuman-1}, the BNM-RF prediction agrees well with the exact solution in terms of both the overall spatial distribution and the oscillatory pattern. The error is mainly concentrated near the boundary and becomes significantly smaller in the interior of the domain. This error distribution is consistent with the typical numerical behavior of boundary integral methods in the near-boundary region, indicating that BNM-RF can accurately capture the interior wave field.

\begin{figure}[htbp]
  \centering
    \begin{subfigure}{0.3\textwidth}
    \includegraphics[width=\linewidth]{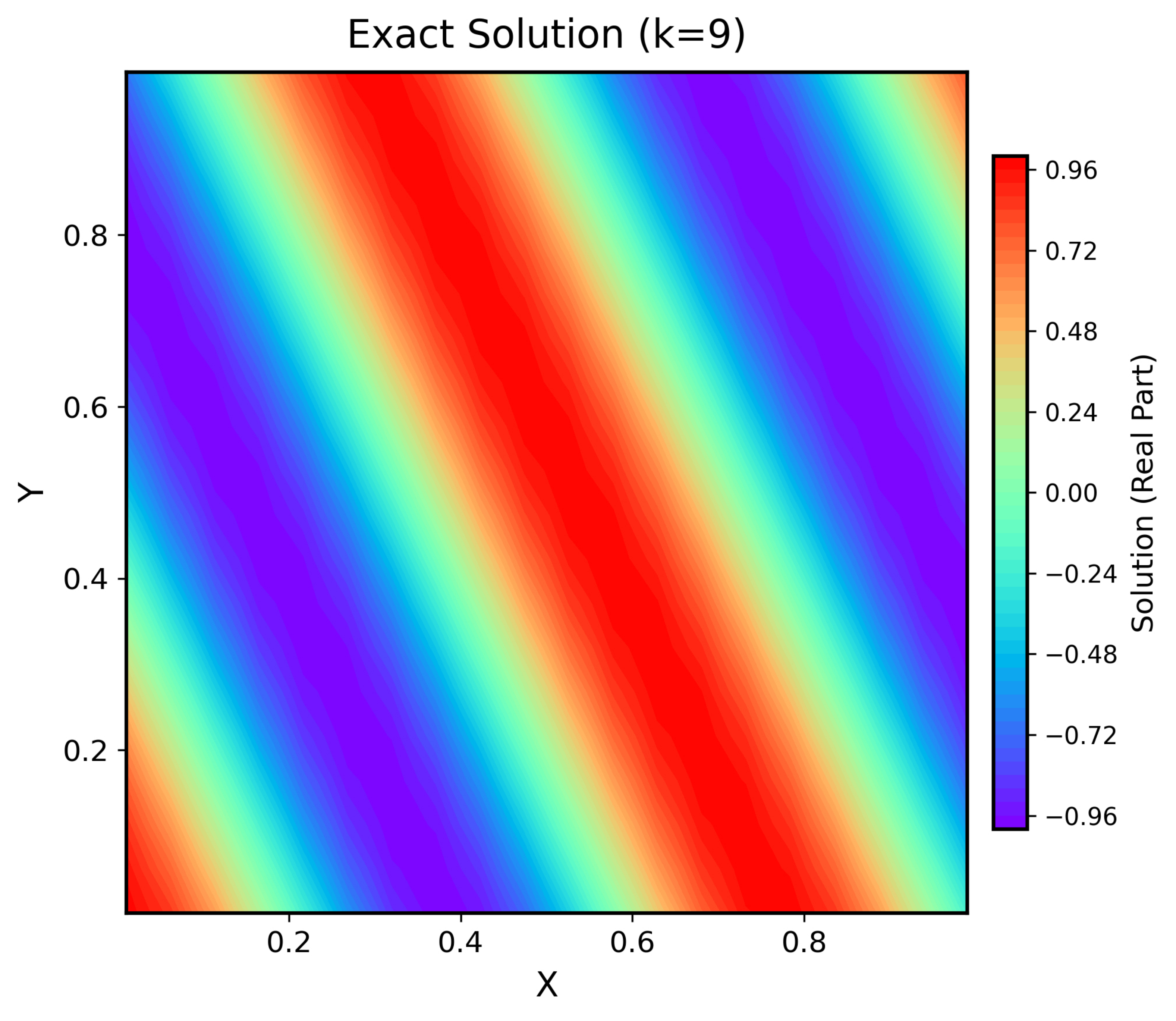}
    \caption{The true solution with k=9 }
  \end{subfigure}
  \hfill
    \begin{subfigure}{0.3\textwidth}
    \includegraphics[width=\linewidth]{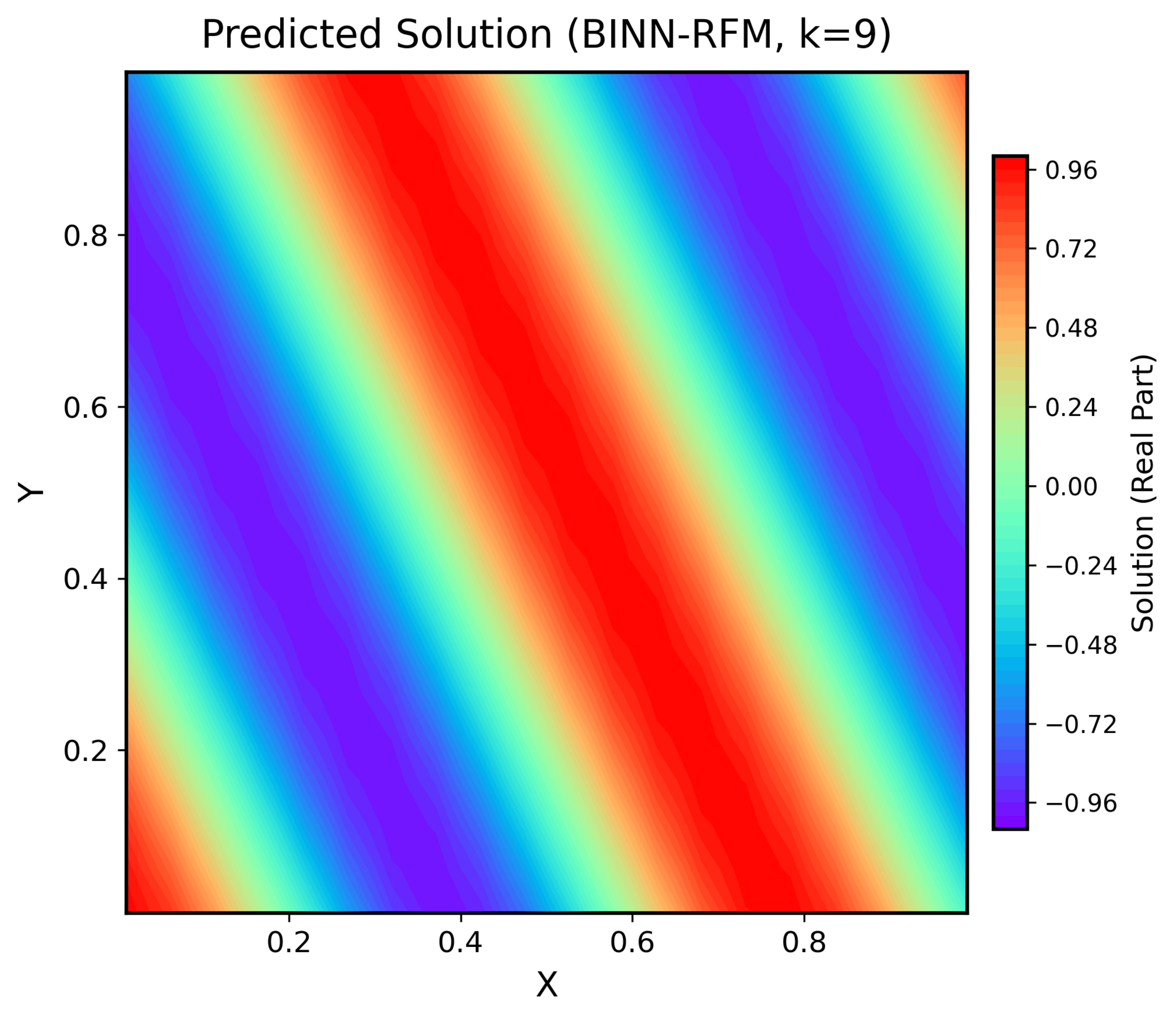}
    \caption{The predicted solution with k=9 }
  \end{subfigure}
  \hfill
    \begin{subfigure}{0.3\textwidth}
    \includegraphics[width=\linewidth]{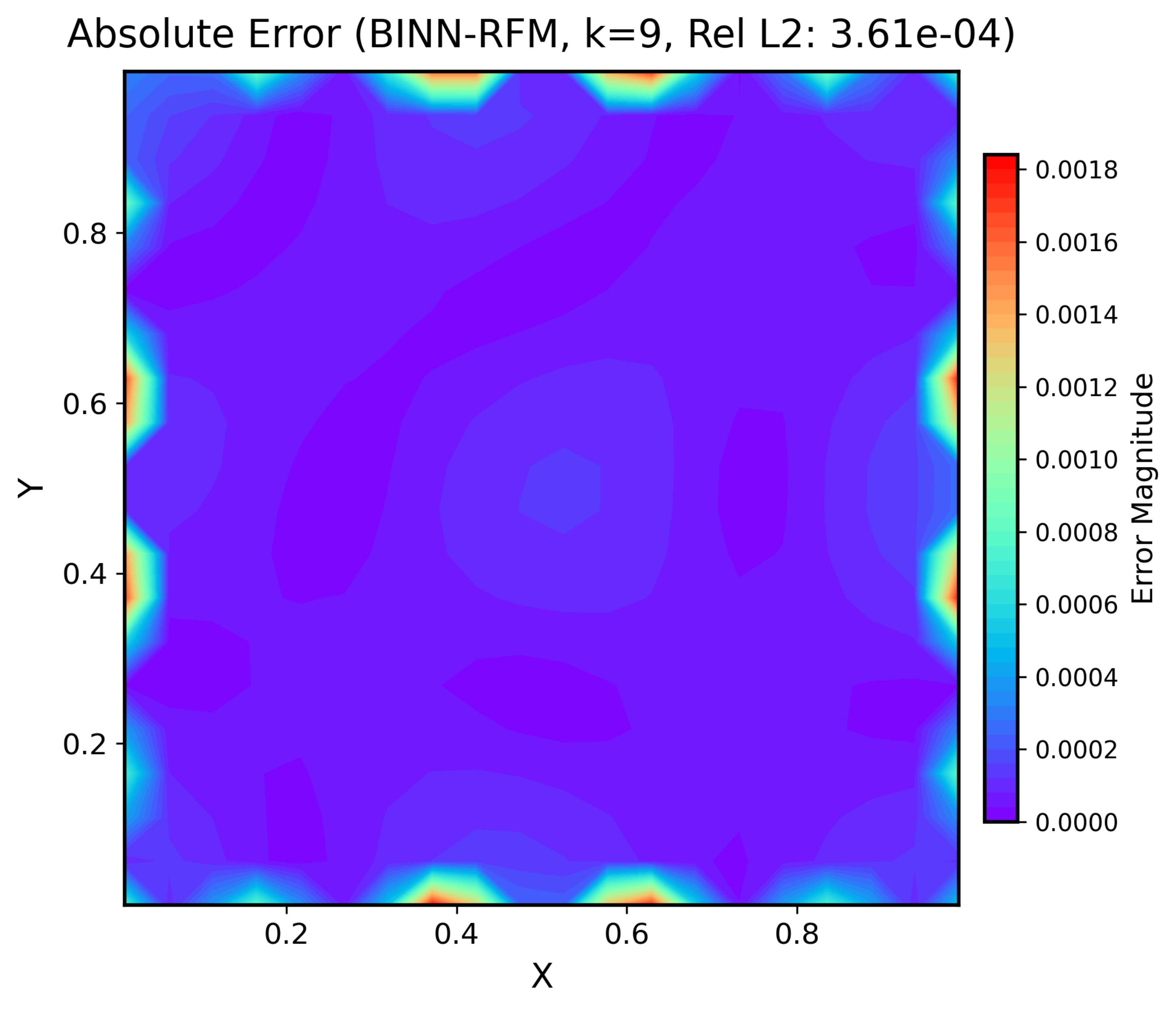}
    \caption{The error plot with k=9 }
  \end{subfigure}
  \caption{Numerical results for the case $k$=9 with $N_s$=60 and $M$=40: (a) true solution; (b) predicted solution; (c) error distribution. }
\label{fig:neuman-1}
\end{figure}

To systematically analyze the convergence rate of the BNM-RF and compare it with the theoretical error bound, we further test the variation of error with increasing number of neurons $M$ and with increasing number of collocation points $N_s$ under different wavenumbers, as shown in Figure \ref{fig:convergence-effect}. The result clearly demonstrate that the actual convergence speed of BNM-RF is significantly faster than the theoretically predicted rate, highlighting the efficiency of the method. Moreover, by comparing with the convergence curves of the conventional boundary element method, BNM-RF requires fewer degrees of freedom to achieve the same accuracy, verifying its performance advantage over traditional boundary element method.

\begin{figure}[htbp]
  \centering
    \begin{subfigure}{0.3\textwidth}
    \includegraphics[width=\linewidth]{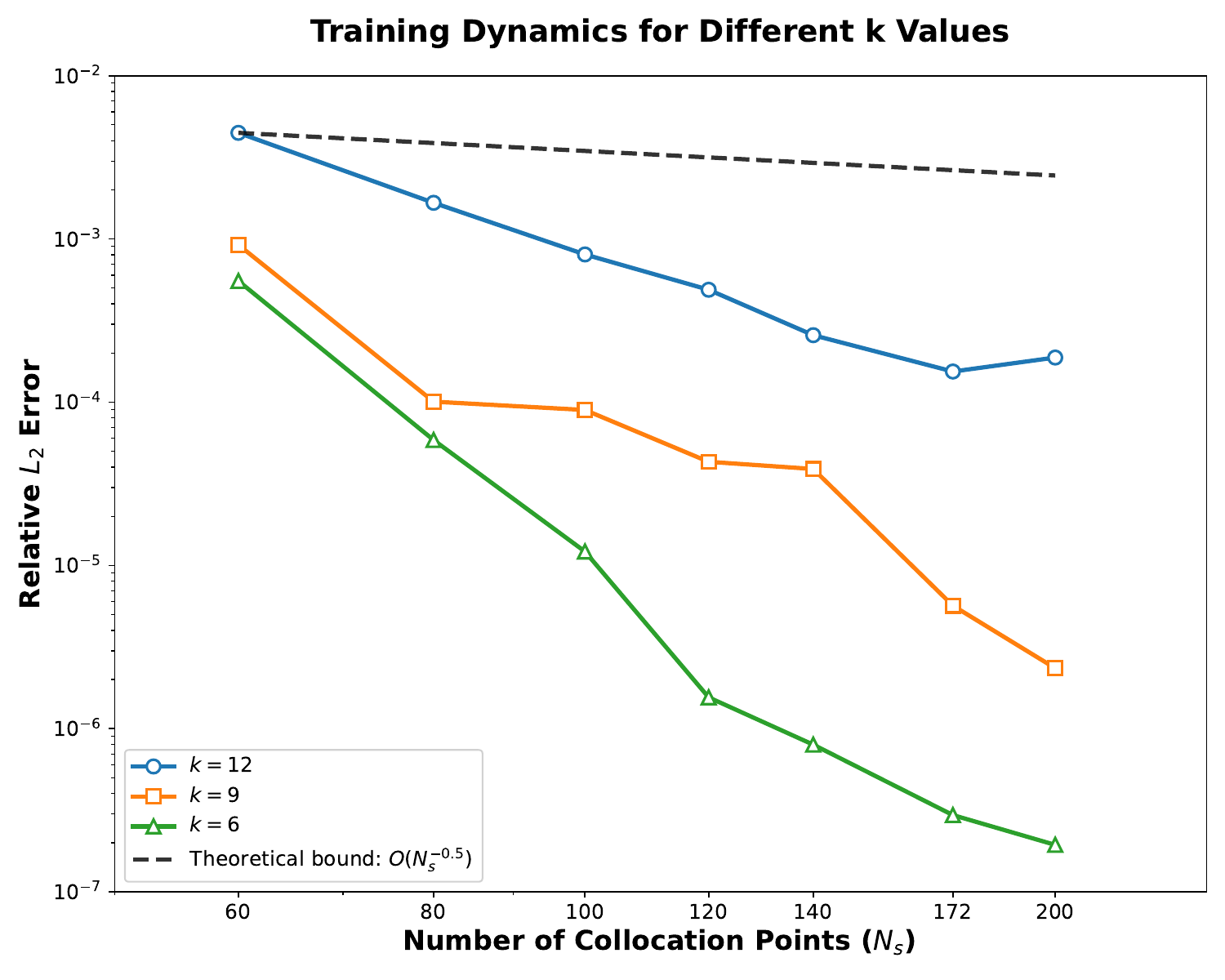}
    \caption{Training Dynamics with collocation points }
  \end{subfigure}
  \hfill
    \begin{subfigure}{0.3\textwidth}
    \includegraphics[width=\linewidth]{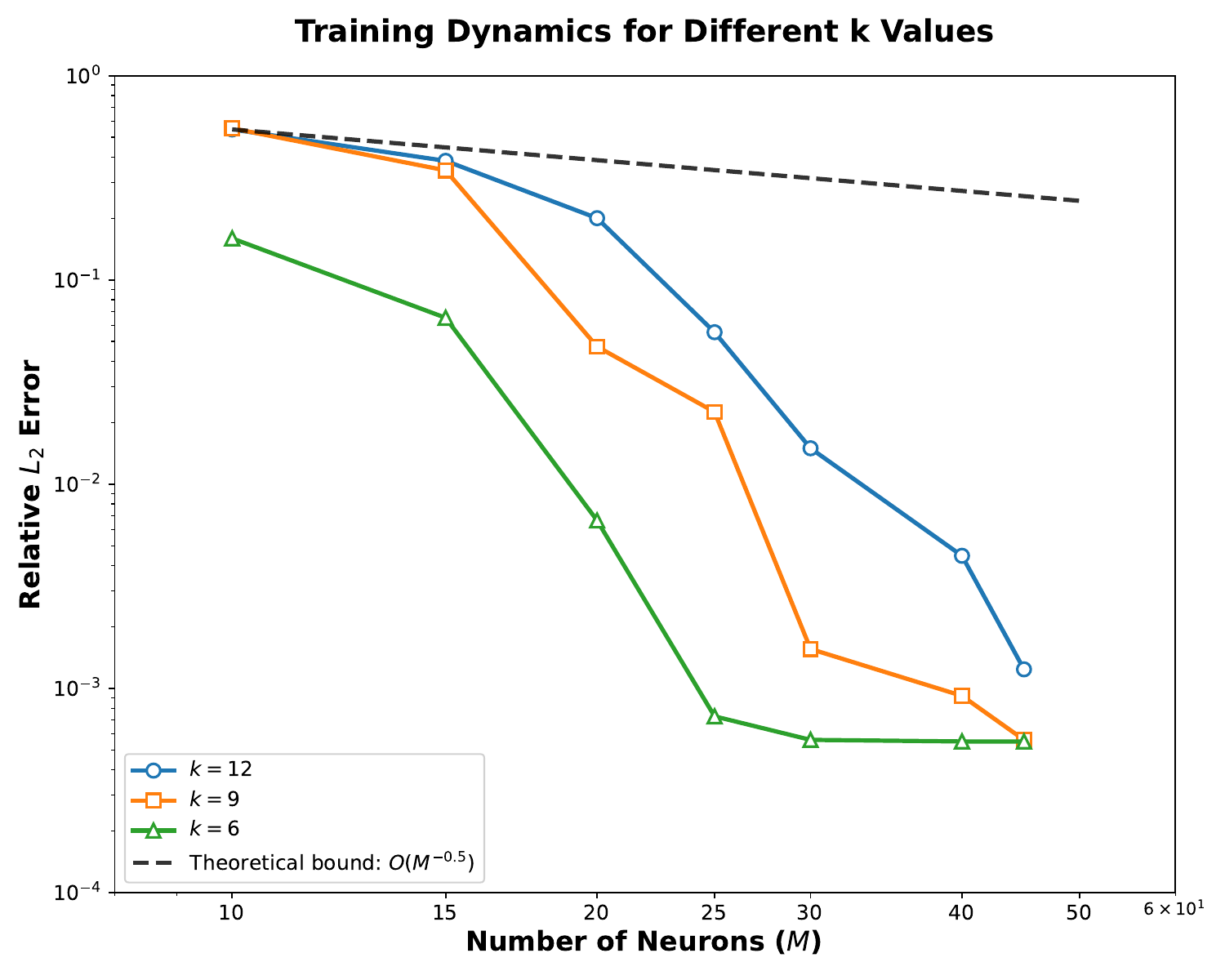}
    \caption{Training Dynamic with neurons }
  \end{subfigure}
  \hfill
    \begin{subfigure}{0.3\textwidth}
    \includegraphics[width=\linewidth]{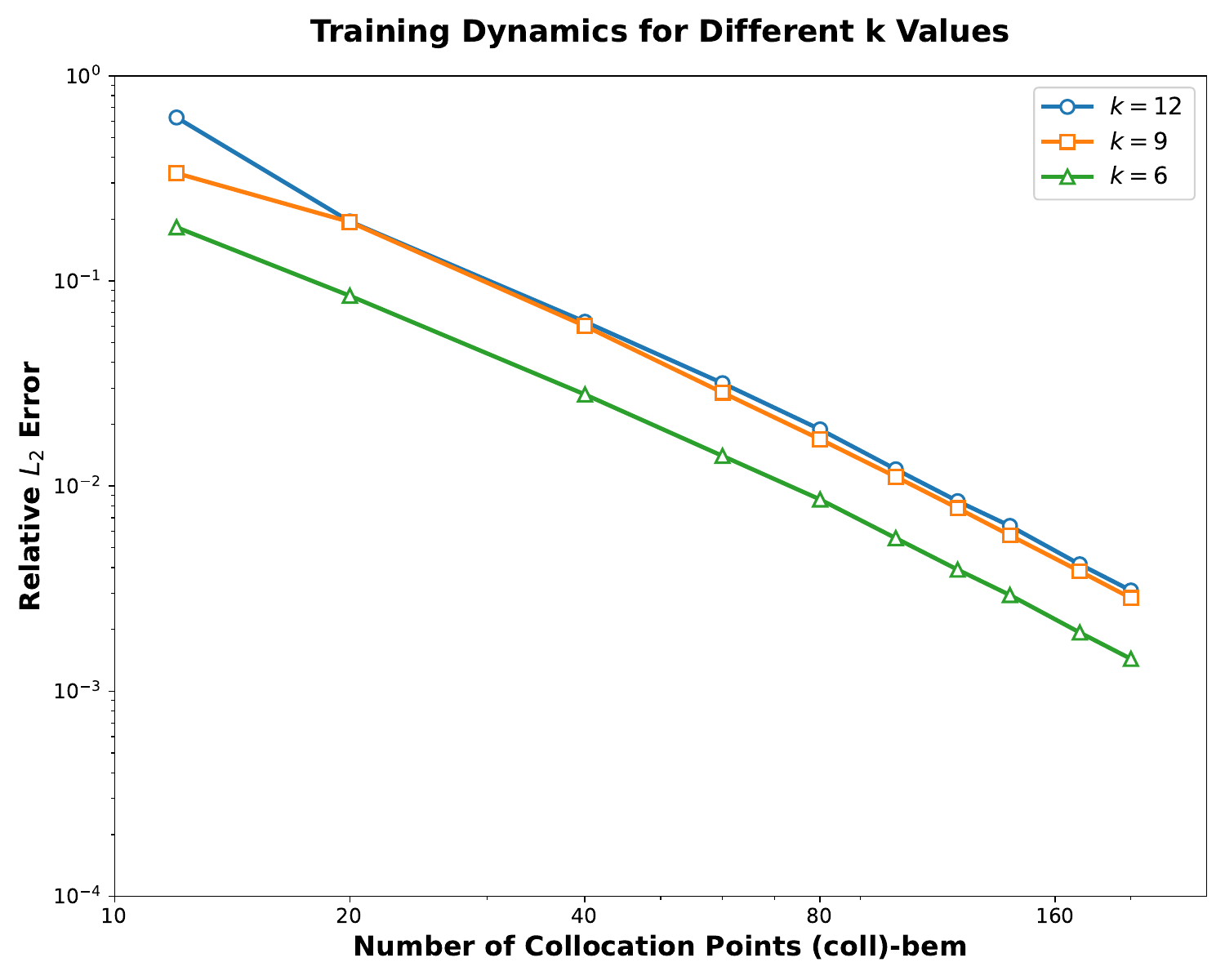}
    \caption{Training Dynamics with bem }
  \end{subfigure}
  \caption{Illustrates the training dynamics through three subplots: (a) variation of the error with the number of collocation points with 40 fixed neurons for different k values; (b) variation of the error with the number of neurons under 60 fixed collocation points for different k values; (c) variation of the error with the number of collocation points for different k values with bem.}
\label{fig:convergence-effect}
\end{figure}

\subsection{Laplace equation}
Since transforming partial differential equations into boundary integral equations only requires integration over the boundary, this approach is well suited for handling problems with complex-shaped boundaries, so we consider the following 2D Poisson problem defined on the flower-shaped region with Dirichlet boundary condition:
\begin{equation}
\label{eq:laplace_eq_1}
\begin{split}
 - \Delta u(\mathbf{x})  = 0 &, \quad \mathbf{x} \in \Omega, \\
            u(\mathbf{x}) = \bar{u}(\mathbf{x})&, \quad \mathbf{x} \in \partial \Omega
\end{split}
\end{equation}
where $\bar{u}(\mathbf{x})$ is given by the analytical solution:
\begin{equation}
\label{eq:laplace_eq_2}
\begin{split}
u(\mathbf{x})=\sin(x_1)\sinh(x_2)+\cos(x_1)\cosh(x_2).
\end{split}
\end{equation}
The flower-shaped region is composed of five semicircles with radius r = 1. For these three methods, 100 collocation points are uniformly allocated on the boundary, and the boundary integrals (including the regular integrals and the regularized singular integrals) are evaluated piecewise using the Gaussian quadrature rule with 10 quadrature points. For the BINN, we configure a fully connected neural network that includes 4 hidden layers, each consisting of 140 neurons. The neural network is trained for 20000 iterations. And for the BNM-RF, this neural network is made up of a hidden layer containing 60 neurons. We show the solution by different methods in Figure~\ref{fig:laplace_sol}, the relative $L^2$ error of these methods is shown in Table \ref{tab:error_comparison}.

\begin{figure}[ht]
    \centering
        \begin{subfigure}[b]{0.32\textwidth}         \centering
        \includegraphics[width=\linewidth]{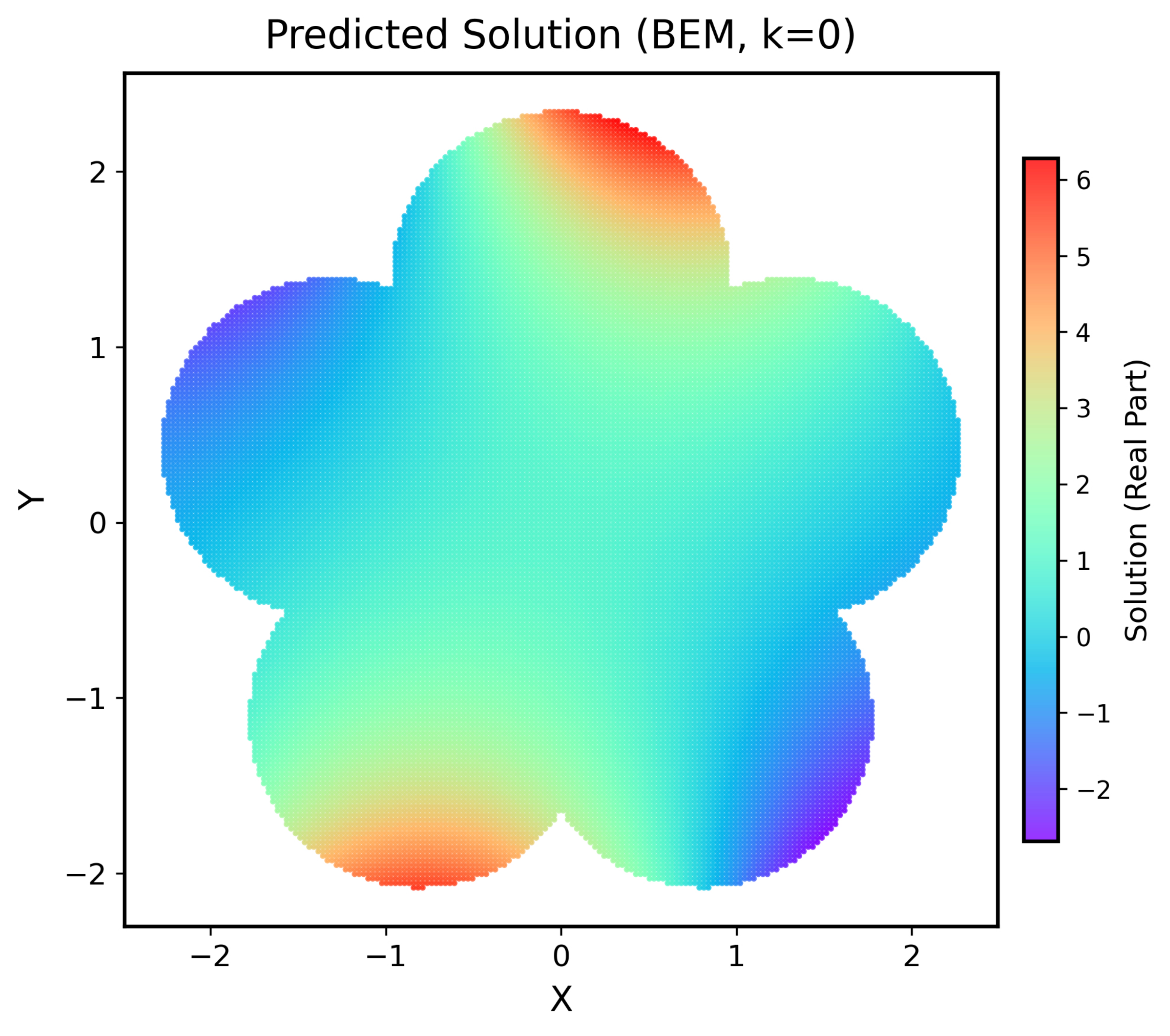}         \caption{Predicted solution by BEM.}
        \label{fig:laplace_flower_bem_pred}
    \end{subfigure}
    \hfill
    \begin{subfigure}[b]{0.32\textwidth}
        \centering
        \includegraphics[width=\linewidth]{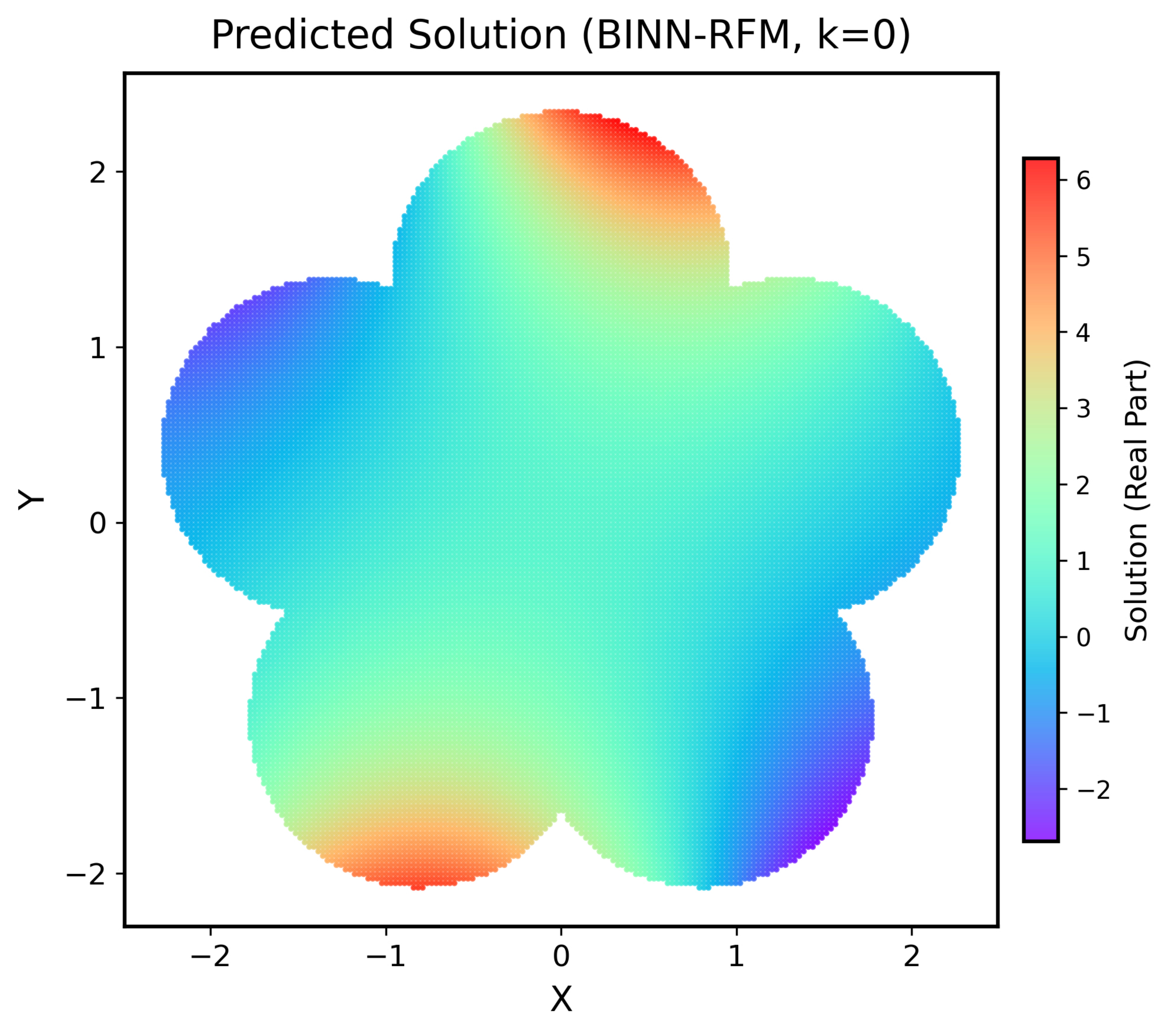}         \caption{Predicted solution by BNM-RF.}
        \label{fig:laplace_flower_bnm_pred}
    \end{subfigure}
    \hfill
    \begin{subfigure}[b]{0.32\textwidth}
        \centering
        \includegraphics[width=\linewidth]{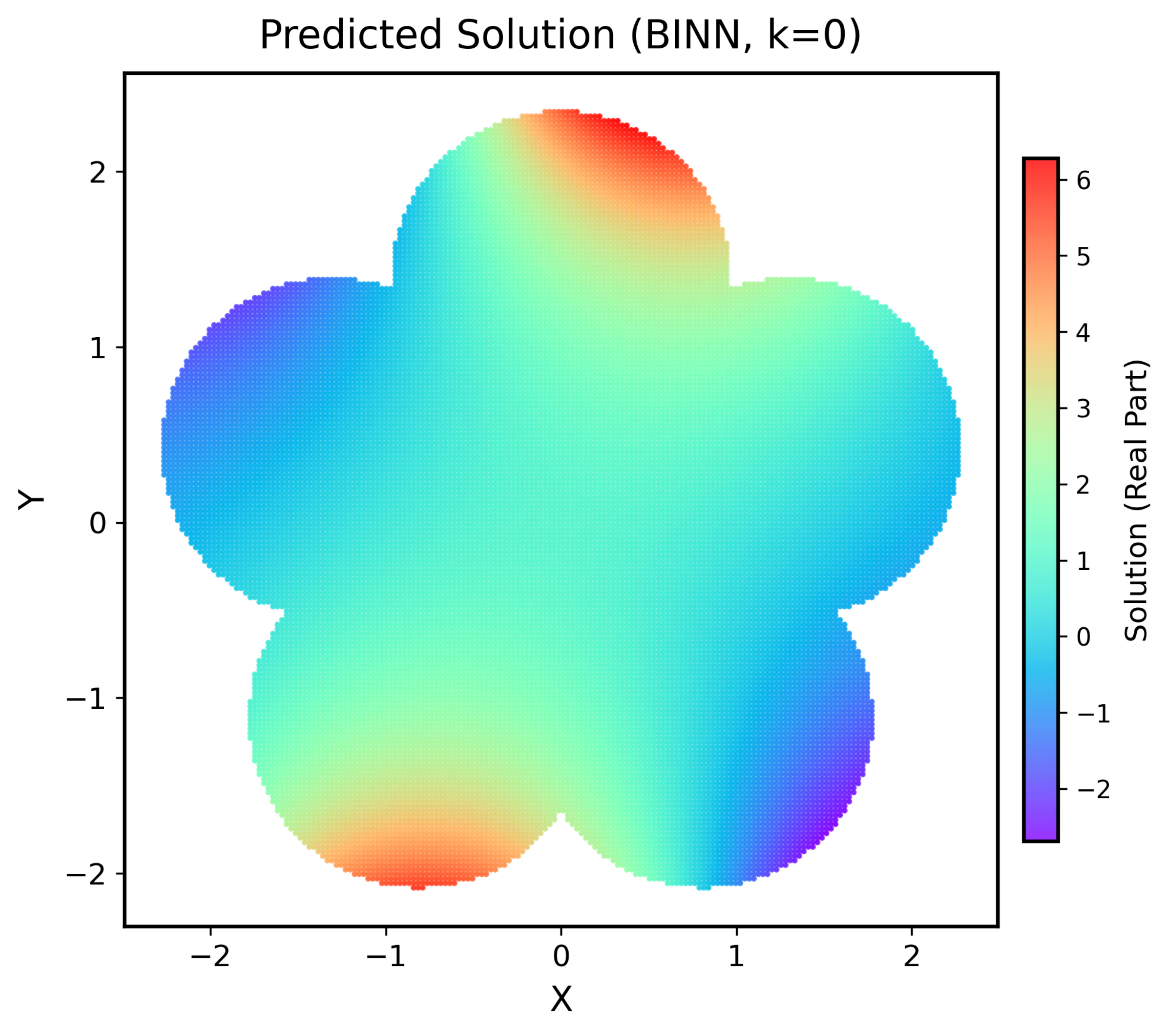}         \caption{Predicted solution by BINN.}
        \label{fig:laplace_flower_binn_pred}
    \end{subfigure}
    
    \vspace{0.5cm}     
        \begin{subfigure}[b]{0.32\textwidth}
        \centering
        \includegraphics[width=\linewidth]{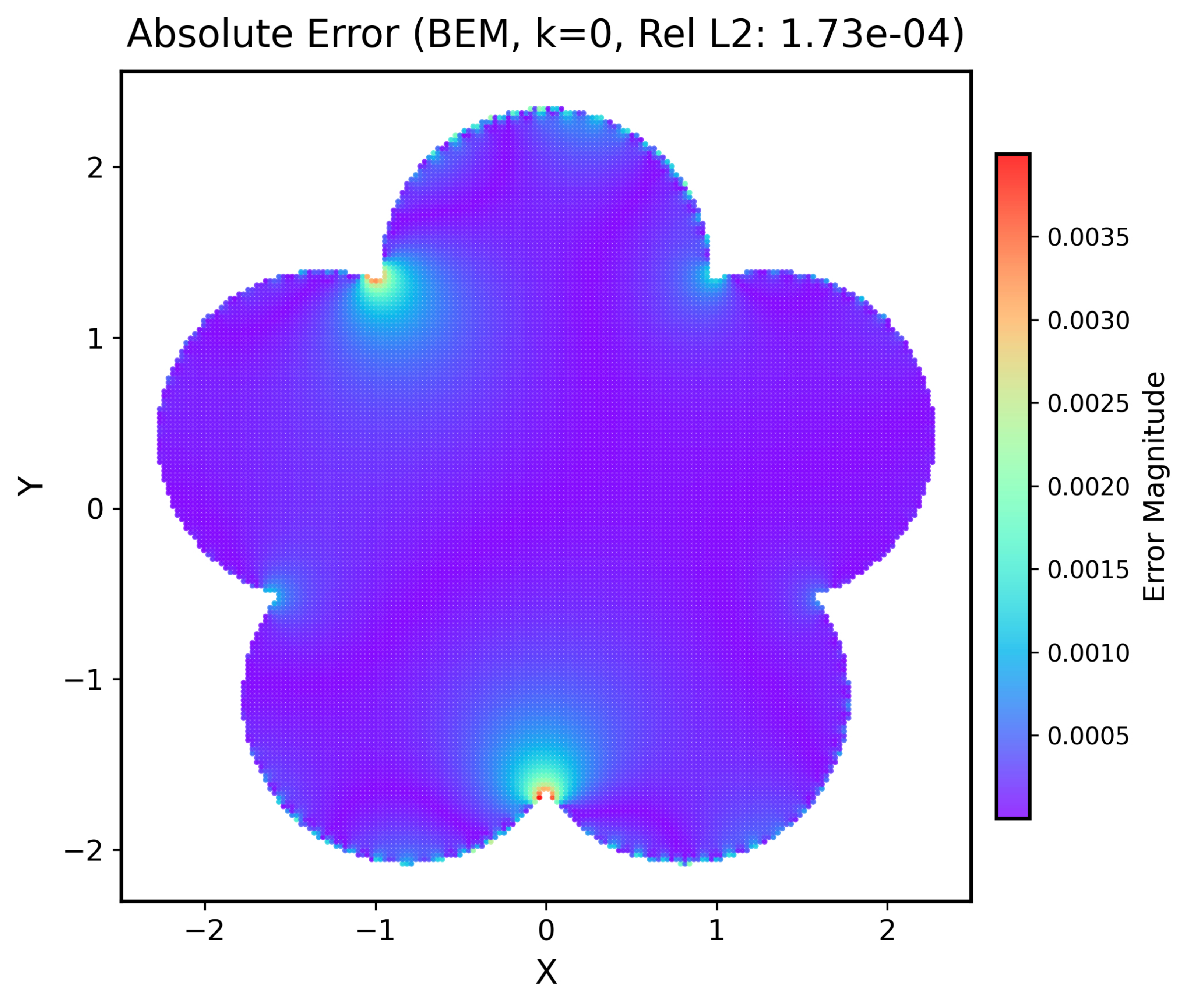}
        \caption{Error plot of the BEM method.}
        \label{fig:flower_bem_error}
    \end{subfigure}
    \hfill
    \begin{subfigure}[b]{0.32\textwidth}
        \centering
        \includegraphics[width=\linewidth]{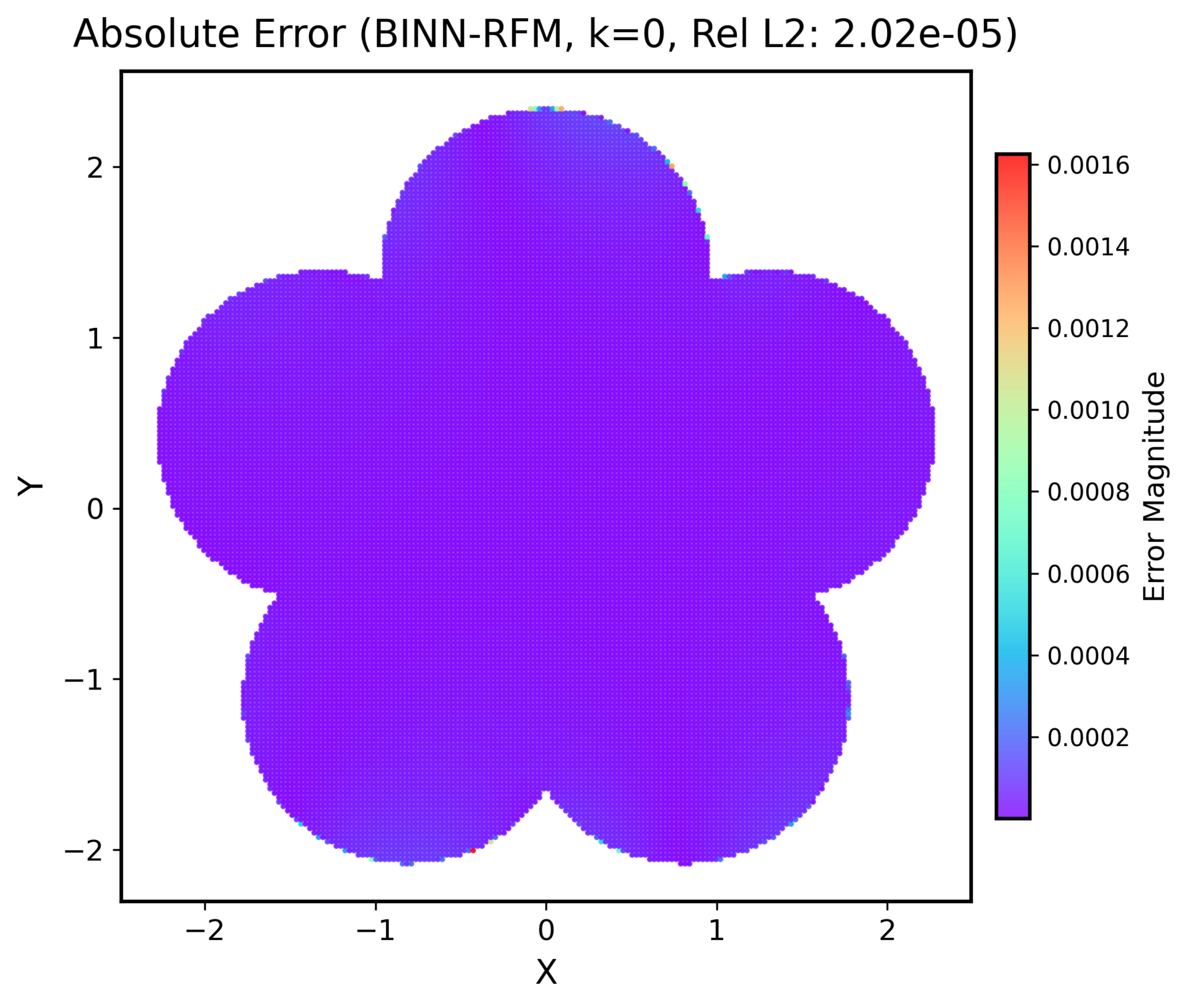}
        \caption{Error plot of the BNM-RF.}
        \label{fig:flower_bnm_error}
    \end{subfigure}
    \hfill
    \begin{subfigure}[b]{0.32\textwidth}
        \centering
        \includegraphics[width=\linewidth]{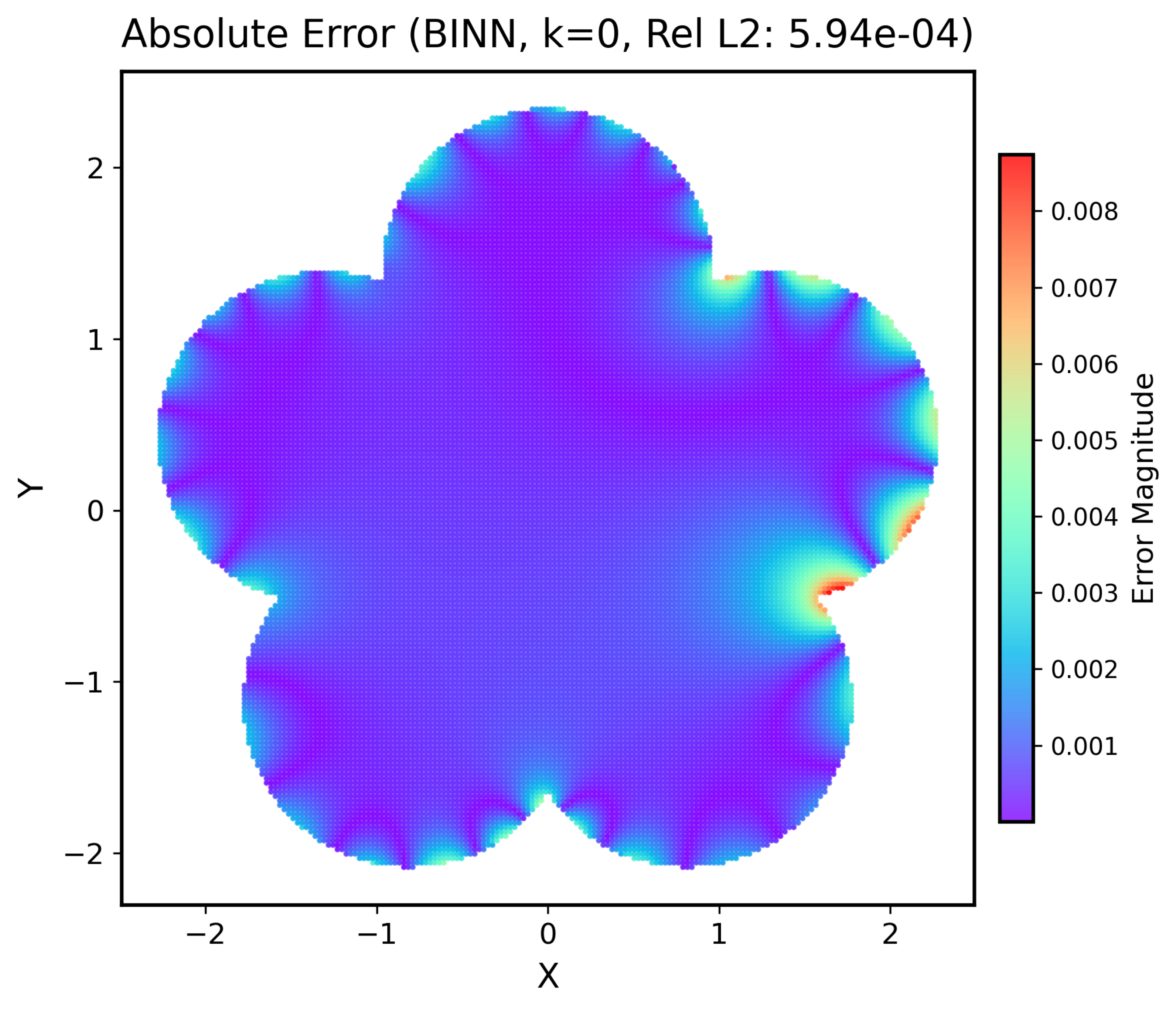}
        \caption{Error plot of the BINN method.}
        \label{fig:flower_binn_error}
    \end{subfigure}
    
    \caption{Comparison of the predicted solutions and error plots of Laplace equation \eqref{eq:laplace_eq_1} with Dirichlet boundary condition by BEM, BNM-RF, and BINN methods.}
    \label{fig:laplace_sol}
\end{figure}
\begin{table}[htbp]
  \centering
  \caption{Relative $L^2$ error of Eq.\eqref{eq:laplace_eq_1} with different methods}
  \label{tab:error_comparison}
  \begin{tabular}{lccc}
    \toprule
    & BINN & BEM & BNM-RF \\
    \midrule
     & $5.94 \times 10^{-4}$   & $1.73 \times 10^{-4}$ & $\mathbf{2.02 \times 10^{-5}}$ \\
    \bottomrule
  \end{tabular}
\end{table}

\begin{table}[htbp]
  \centering
  \caption{Relative $L^2$ error of Eq.\eqref{eq:laplace_eq_1} with different random initialization parameters}
  \label{tab:error_comparison_lap_uniform}
  \begin{tabular}{lccc}
    \toprule
   tanh & $[-1,1]$ & $[-2,2]$ & $[-4,4]$ \\
    \midrule
    & $2.02 \times 10^{-5}$ &  $1.42 \times 10^{-5}$& $4.78 \times 10^{-5}$ \\
    \midrule
   cos & $\gamma=1.0$ & $\gamma=2.0$ & $\gamma=4.0$ \\
    \midrule
    & $1.39 \times 10^{-5}$ & $1.39 \times 10^{-5}$ & $1.40 \times 10^{-5}$ \\
    \bottomrule
  \end{tabular}
\end{table}

It is observed that the BNM-RF achieves superior performance compared to BINN under the condition of fewer trainable parameters. This improvement is mainly due to the use of the least squares method in BNM-RF to efficiently determine the connection weights between the hidden layer and the output layer. In contrast, BINN relies on gradient descent-based iterative optimization, which may suffer from slow convergence, sensitivity to local minima, and long training time. As a result, BNM-RF shows better training efficiency and stability.

Table \ref{tab:error_comparison_lap_uniform} further investigates the robustness of the proposed BNM-RF with respect to the initialization of random parameters. For the $\tanh$ activation function, we test different ranges for the weight initialization, while for the cosine activation function, we vary the parameter $\gamma$. The results indicate that the performance of BNM-RF remains stable and consistently accurate  across a range of initialization settings, demonstrating its low sensitivity to these hyperparameters.

\subsection{Exterior problem for the Helmholtz equation}

As a further example, we consider the 2D exterior Helmholtz problem with Dirichlet boundary condition and Sommerfeld radiation condition:
    \begin{equation} 
        \label{eq:origin_helmholtz}
        \begin{split}
            - \Delta u(\mathbf{x}) - k^2 u(\mathbf{x}) = 0 &, \quad \mathbf{x} \in \Omega^c, \\
            u(\mathbf{x}) = \bar{u}(\mathbf{x})&, \quad \mathbf{x} \in \partial \Omega, \\
            \lim_{|\mathbf{x}| \to \infty} (\frac{\partial}{\partial r} - i k) u(\mathbf{x}) = o(|\mathbf{x}|^{-1/2})
        \end{split}
    \end{equation}
where $\bar{u}(\mathbf{x})$ is given by the analytical solution:    
\begin{equation}
\label{eq:helmholt_eq_1}
u(\mathbf{x})=H_0^{(1)}(k\sqrt{x_1^2+x_2^2})
\end{equation}
This study considers an unbounded domain with a star-shaped region as the interior boundary, denoted by $\Omega^c$. A dirichlet boundary condition is imposed on its boundary $\partial\Omega$. To systematically evaluate the performance of different numerical methods, we perform a series of experiments with varying wave numbers $k$. Three methods—BINN, BEM and BNM-RF are employed to solve the problem.

In the numerical implementation, all three methods uniformly distribute 100 collocation points on the boundary and employ the 10-point Gaussian quadrature rule to evaluate the boundary integrals. For the BINN, we employ a fully connected  neural network architecture comprising 4 hidden layers with 80 neurons each, trained for 12000 iterations. For the BNM-RF, the network contains only a single hidden layer with 60 neurons. We show the solution by different methods in Figure~\ref{fig:helmholtz_sol_dirichlet} and the relative $L^2$ error of these methods is shown in Table \ref{tab:error_comparison_helm}.

\begin{figure}[ht]
    \centering
        \begin{subfigure}[b]{0.32\textwidth}
        \centering
        \includegraphics[width=\linewidth]{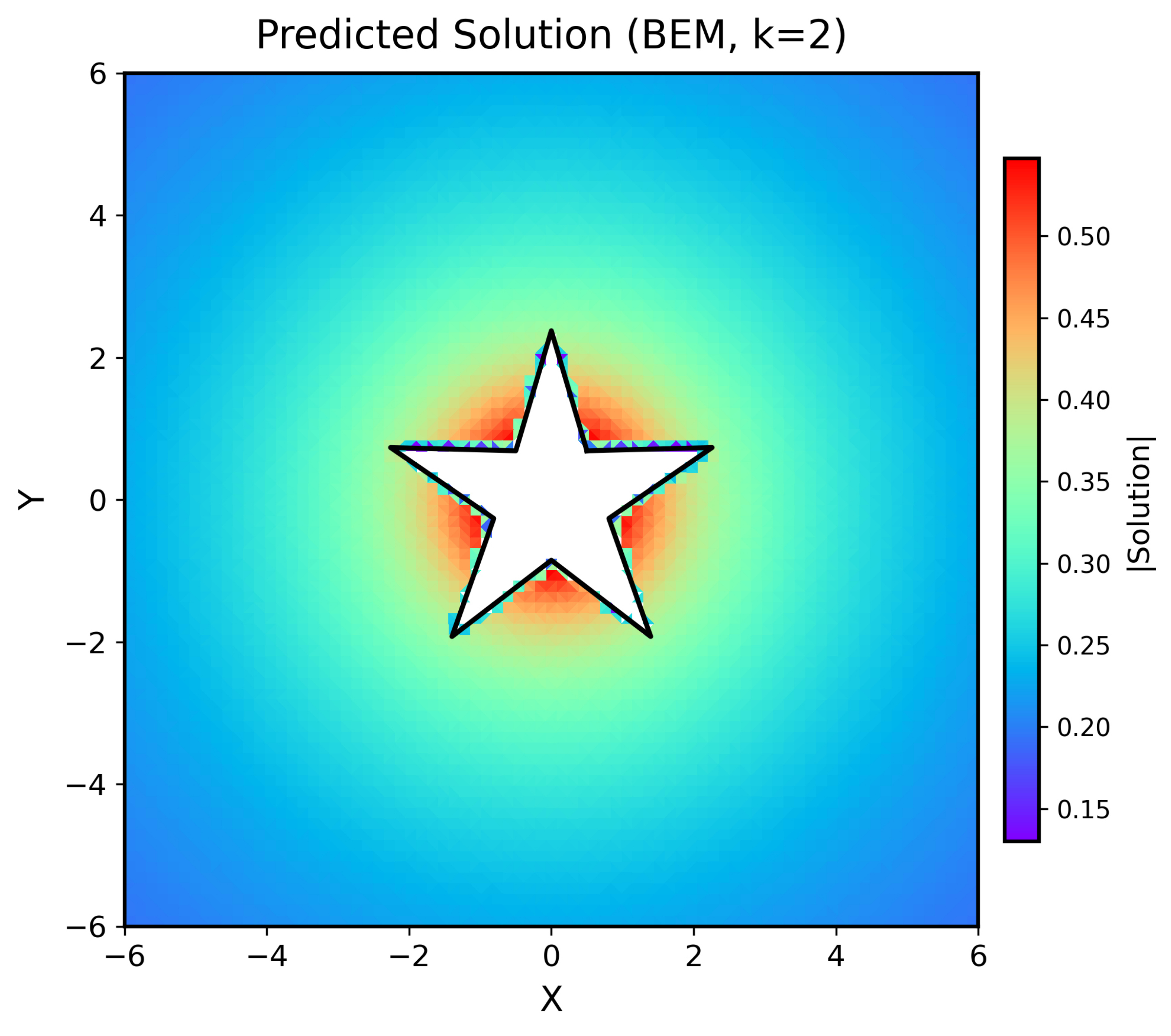}
        \caption{BEM predicted solution.}
        \label{fig:bem_pred}
    \end{subfigure}
    \hfill
    \begin{subfigure}[b]{0.32\textwidth}
        \centering
        \includegraphics[width=\linewidth]{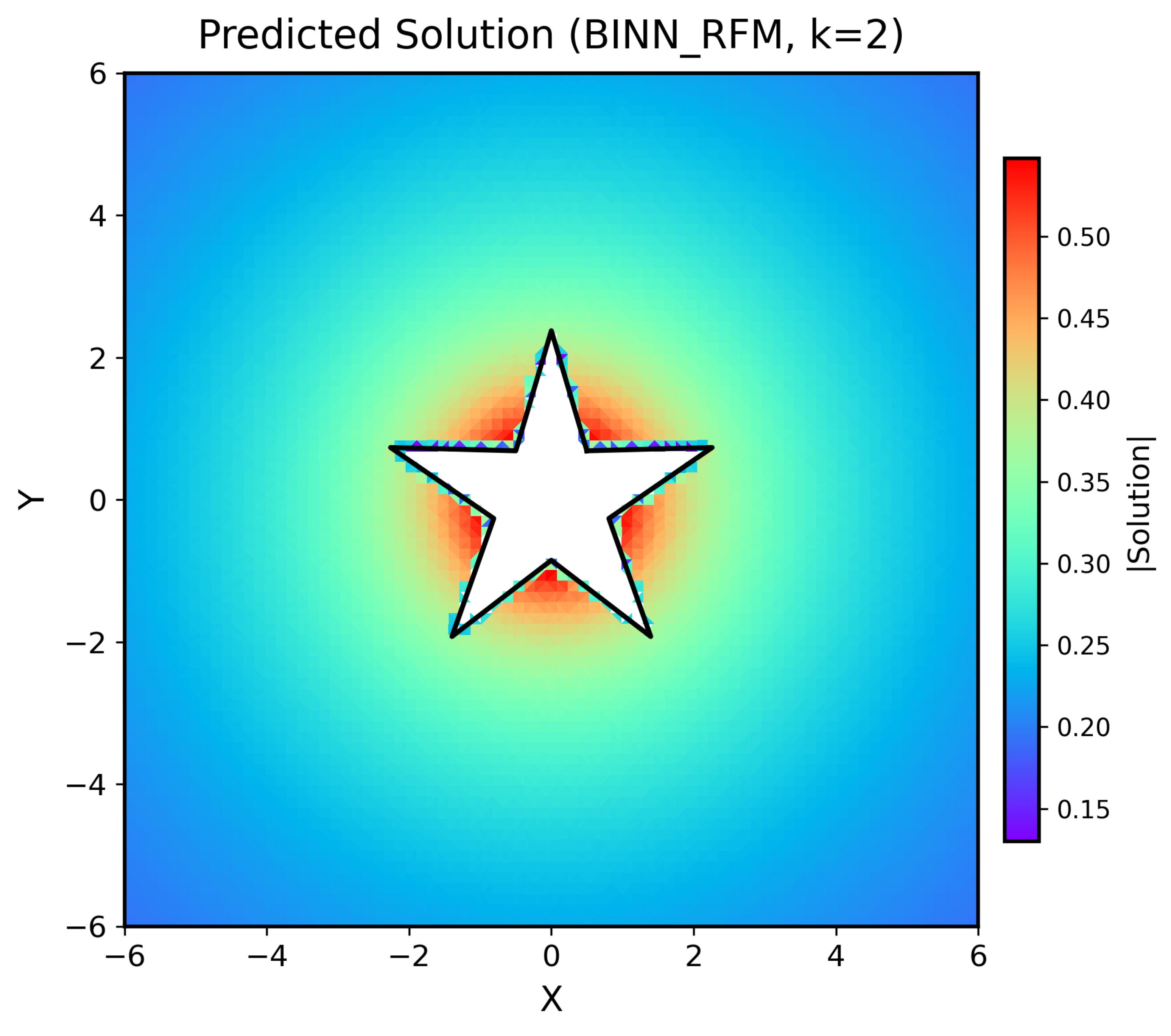}
        \caption{BNM-RF predicted solution.}
        \label{fig:bnm_pred}
    \end{subfigure}
    \hfill
    \begin{subfigure}[b]{0.32\textwidth}
        \centering
        \includegraphics[width=\linewidth]{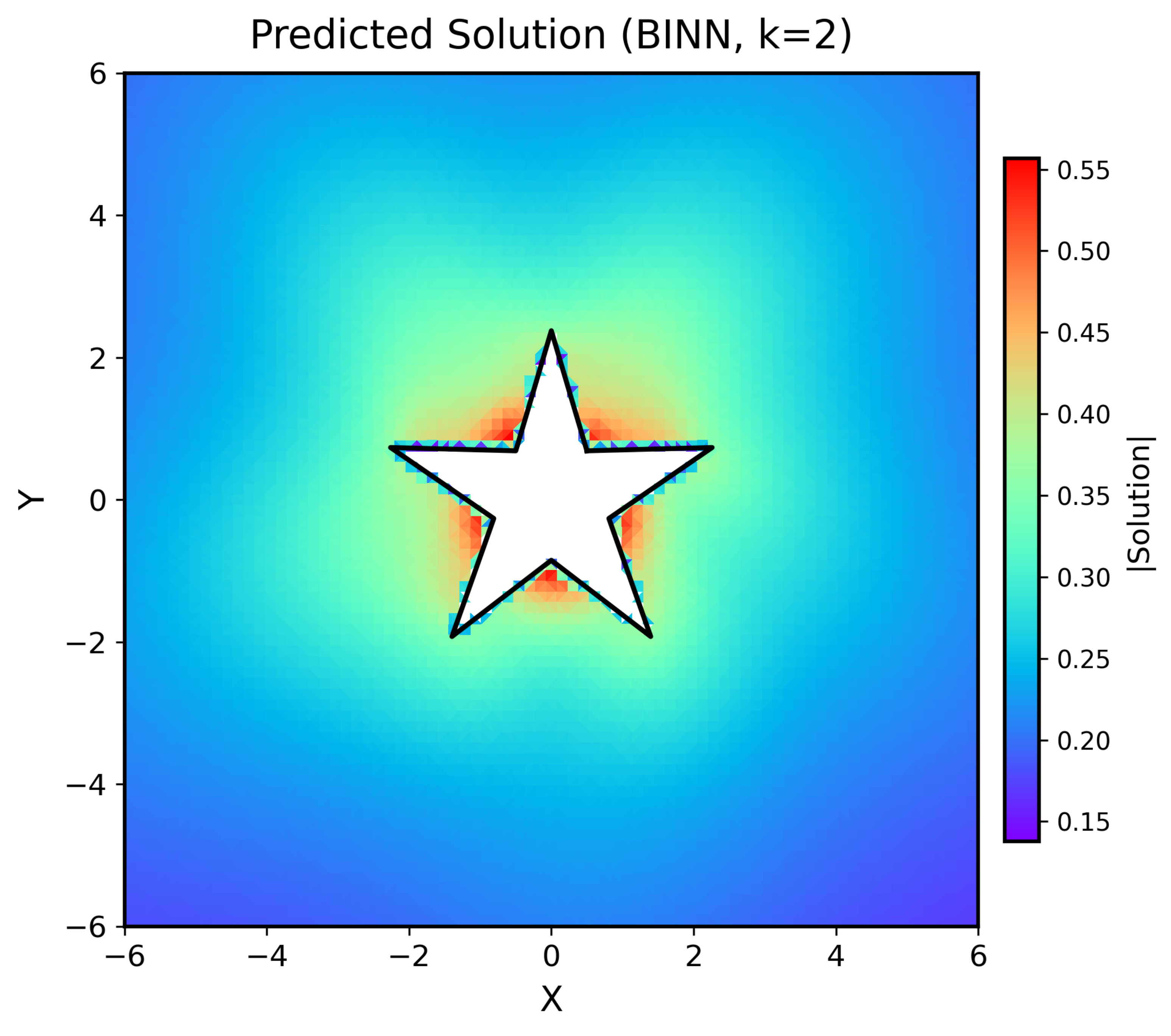}
        \caption{BINN predicted solution.}
        \label{fig:binn_pred}
    \end{subfigure}
    
    \vspace{0.5cm}
    
        \begin{subfigure}[b]{0.32\textwidth}
        \centering
        \includegraphics[width=\linewidth]{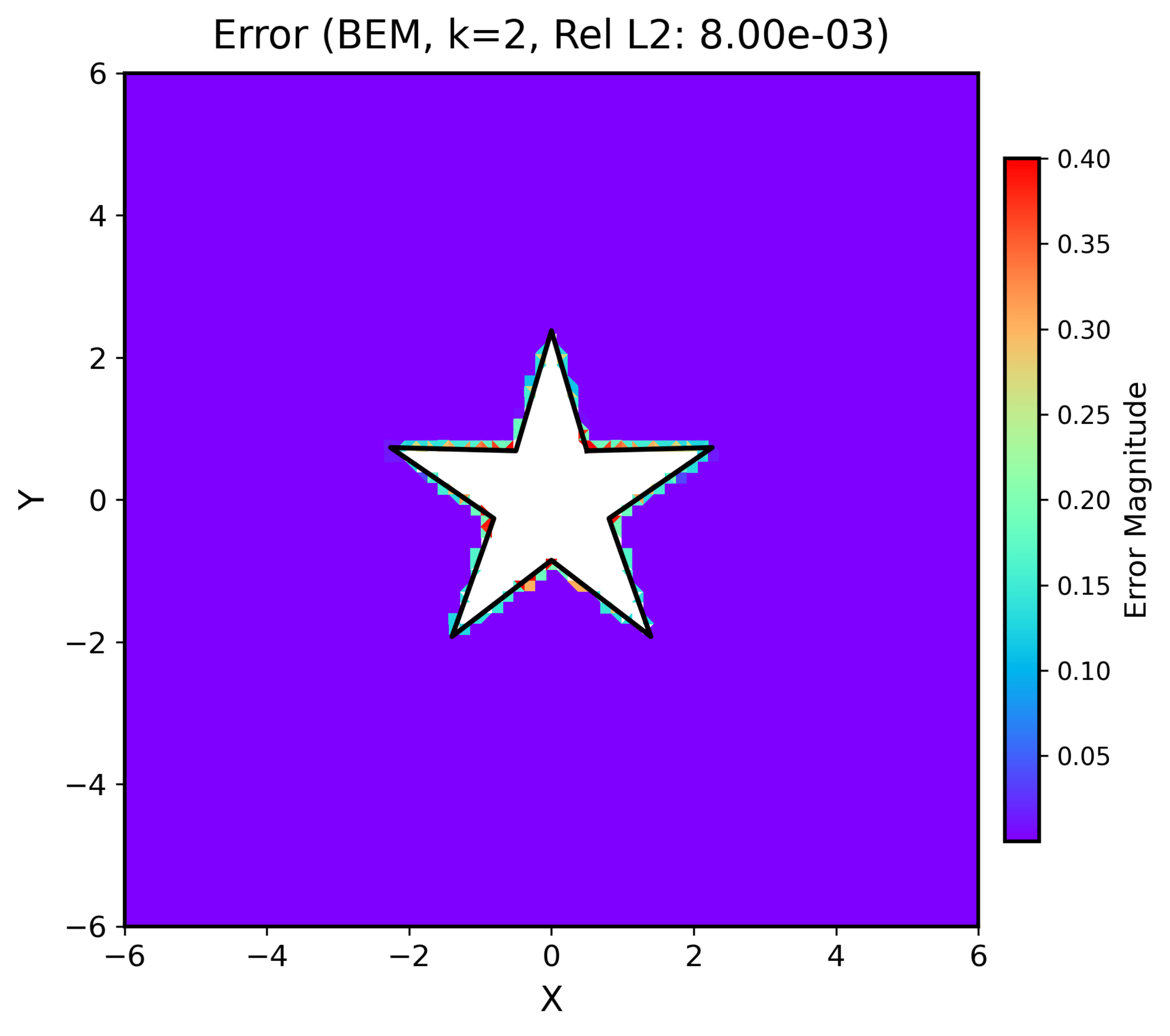}
        \caption{BEM error plot.}
        \label{fig:bem_error}
    \end{subfigure}
    \hfill
    \begin{subfigure}[b]{0.32\textwidth}
        \centering
        \includegraphics[width=\linewidth]{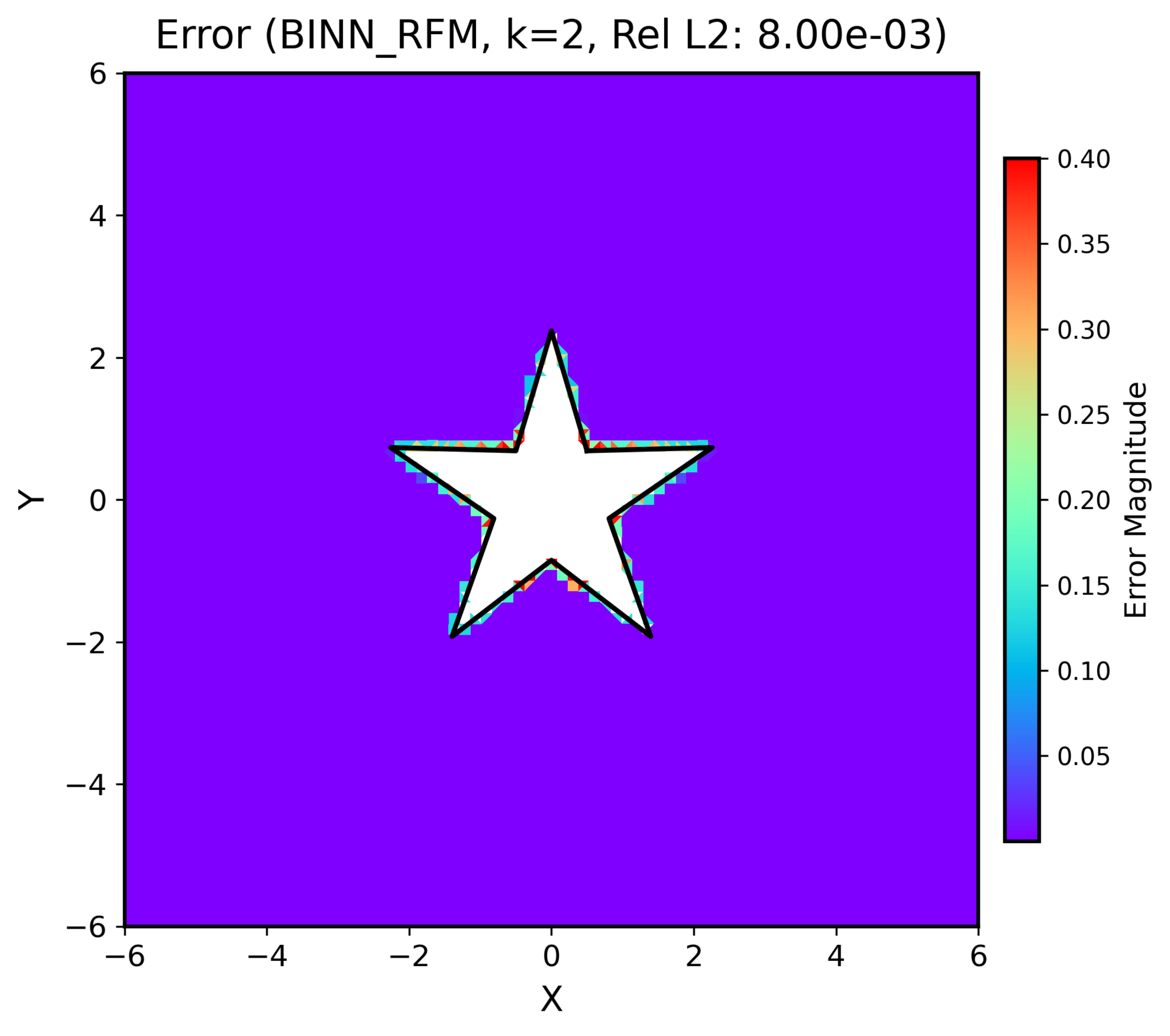}
        \caption{BNM-RF error plot.}
        \label{fig:bnm_error}
    \end{subfigure}
    \hfill
    \begin{subfigure}[b]{0.32\textwidth}
        \centering
        \includegraphics[width=\linewidth]{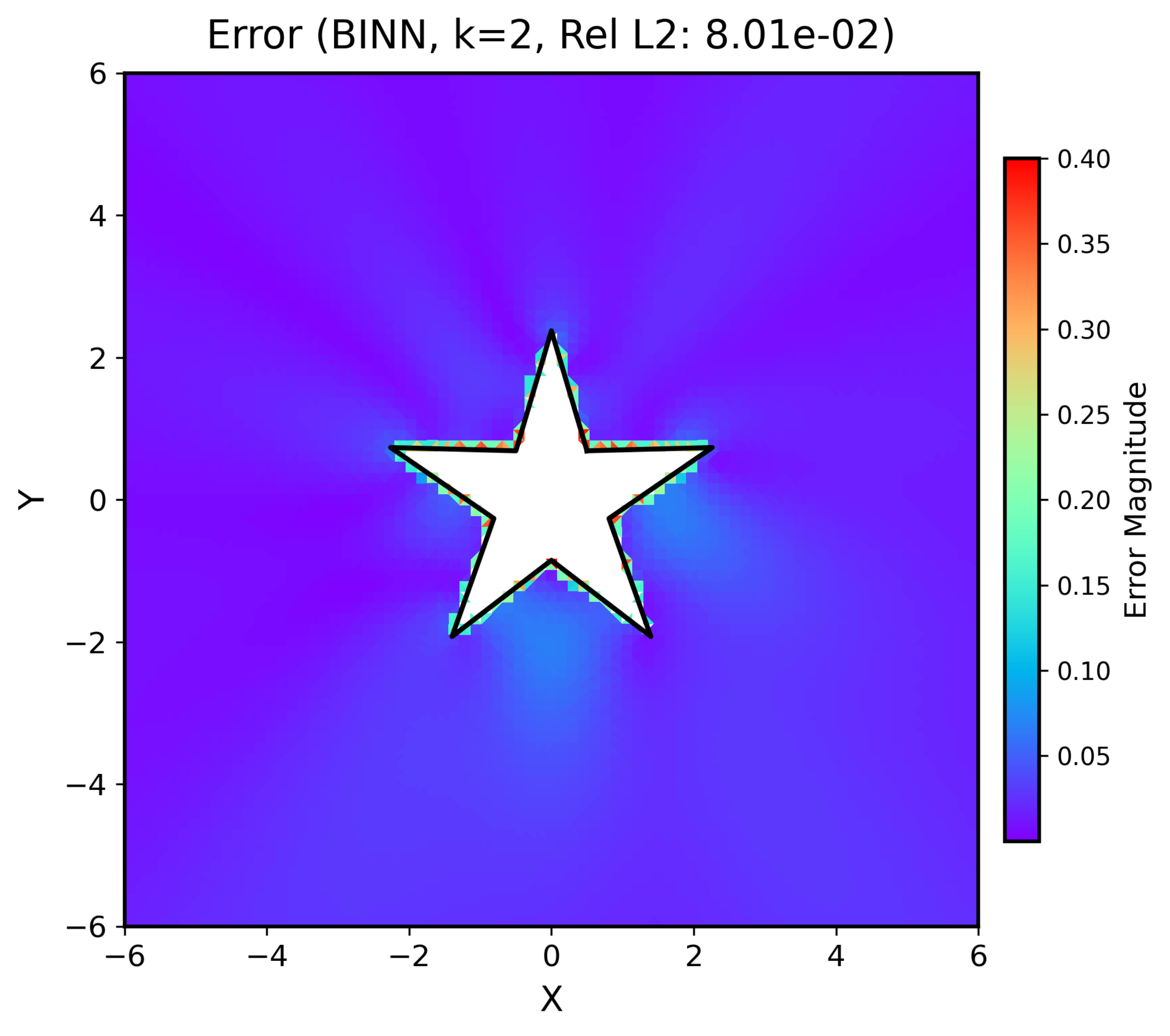}
        \caption{BINN error plot.}
        \label{fig:binn_error}
    \end{subfigure}

\caption{Numerical results for Helmholtz equation \eqref{eq:origin_helmholtz} with $k=2$ and Dirichlet boundary condition.}
\label{fig:helmholtz_sol_dirichlet}
\end{figure}

\begin{table}[htbp]
  \centering
  \caption{Relative $L^2$ error of Eq.\eqref{eq:origin_helmholtz} with different methods}
  \label{tab:error_comparison_helm}
  \begin{tabular}{lccc}
    \toprule
    & BINN & BEM & BNM-RF \\
    \midrule
    $k=2$ & $8.01\times 10^{-2}$ & $8.00\times 10^{-3}$ & $7.99\times 10^{-3}$ \\
    $k=4$ & $9.78\times 10^{-2}$ & $8.10\times 10^{-3}$ & $8.02\times 10^{-3}$ \\
    $k=6$ & $1.21\times 10^{-1}$ & $8.87\times 10^{-3}$ & $8.91\times 10^{-3}$ \\
        \bottomrule
  \end{tabular}
\end{table}

As shown in Table \ref{tab:error_comparison_helm}, the relative $L^2$ errors of BNM-RF and BEM are consistently much smaller than those of BINN for all tested wave numbers. Moreover, BNM-RF achieves accuracy comparable to that of BEM at low and moderate wave numbers. These results indicate that BNM-RF is an effective method for solving the present exterior Helmholtz problem and performs significantly better than BINN in terms of accuracy.

\subsection{Acoustic scattering problem}
\label{sec:acoustic-scattering}

We consider a classical benchmark in acoustic scattering: a plane wave incident on a sound-hard sphere in an unbounded domain. This problem is widely used to validate boundary-type numerical methods such as the boundary element method (BEM). The computational details, including the Burton--Miller formulation \cite{kirkup2019acoustics}, surface discretization, collocation strategy, and numerical quadrature, are provided in Appendix~\ref{app:operators}. Due to the relatively slow iterative training of BINN, it is not included in the present experiment.

The sphere is assumed to be sound-hard, that is,
$$
\frac{\partial \phi}{\partial \mathbf{n}}=0,
\qquad \forall \mathbf{x}\in\partial\Omega.
$$

A plane wave is incident on the sphere in the \(z\)-direction. The incident acoustic potential is given by
\begin{equation}
\phi_{\mathrm{inc}}(\mathbf{x}) = A e^{\mathrm{i}\mathbf{k}\cdot\mathbf{x}}
\end{equation}
where \(A\) is the wave amplitude and \(\mathbf{k}=k\mathbf{d}\), with \(|\mathbf{d}|=1\), is the wave vector in the direction \(\mathbf{d}\). In this experiment, \(\mathbf{d}=\{0,0,1\}\) and \(A=1\). The total acoustic potential is defined by
\begin{equation}
\phi=\phi_s+\phi_{\mathrm{inc}}
\end{equation}
where \(\phi_s\) denotes the scattered field.

For this benchmark, the analytical solution  of the scattered field can be found in \cite{1969Theoretical} and is repeated here for convenience:
\begin{equation}
\phi_s(r,\theta)=\sum_{n=1}^{\infty}
-\frac{\mathrm{i}^n(2n+1)j_n'(ka)}{h_n'(ka)}\,P_n(\cos\theta)\,h_n(kr). \tag{42}
\end{equation}
Here \(j_n\) denotes the spherical Bessel function of the first kind of order \(n\), \(h_n\) denotes the spherical Hankel function of the first kind of order \(n\), \(P_n\) denotes the Legendre polynomial of degree \(n\), and \(a\) is the radius of the sphere. The symbols \(j_n'\) and \(h_n'\) denote derivatives with respect to their arguments.

Because the scattered field is oscillatory, a sufficiently fine discretization is required to achieve accurate numerical resolution. The BNM-RF is tested with different numbers of neurons. The numerical solutions obtained by different methods are shown in Figure \ref{fig:plane-wave-impinged}, and the relative \(L^2\) errors are reported in Table \ref{tab:error_comparison_scatter}.

Figure \ref{fig:plane-wave-impinged} shows that both BEM and BNM-RF reproduce the main scattering pattern of the acoustic field for \(k=8\), with good overall agreement with the exact solution. However, BNM-RF exhibits slightly larger local deviations near the sphere. Table \ref{tab:error_comparison_scatter} shows that the relative \(L^2\) errors increase with the wavenumber for both methods. Although BEM is more accurate, the errors of BNM-RF remain comparable in magnitude. The lower accuracy of BNM-RF, as well as the lack of improvement when the number of neurons increases from 32 to 64, may be mainly attributed to the insufficient treatment of singular integrals in the present implementation. This phenomenon may also be related to the geometric approximation error introduced by the triangular mesh approximation of the sphere.

  \begin{figure}[ht]
\centering
\begin{subfigure}{0.3\textwidth}
    \centering
    \includegraphics[width=\linewidth]{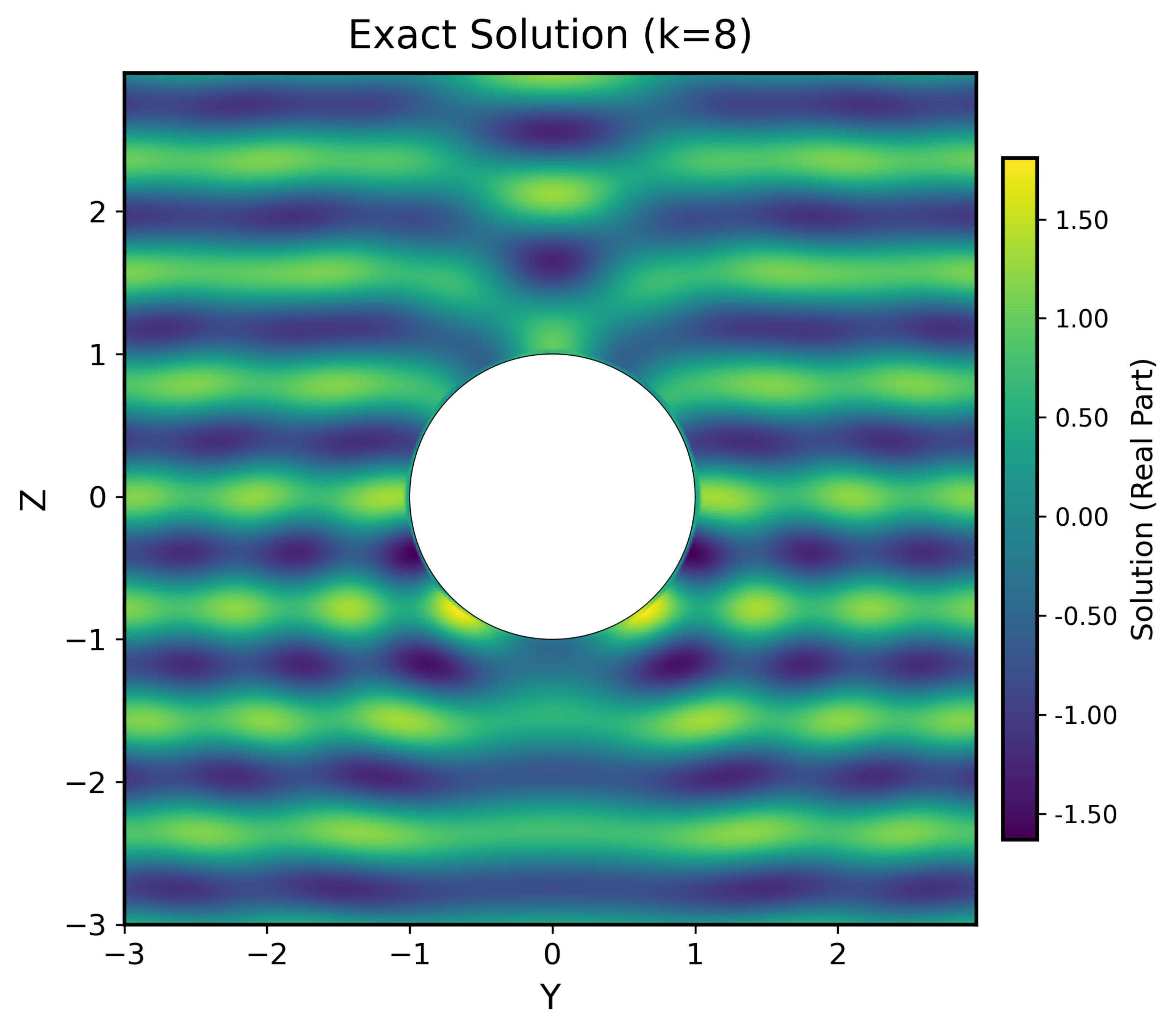}
    \caption{Exact solution on the Y-Z plane}
    \label{subfig:exact_solution}
\end{subfigure}
\hfill
\begin{subfigure}{0.3\textwidth}
    \centering
    \includegraphics[width=\linewidth]{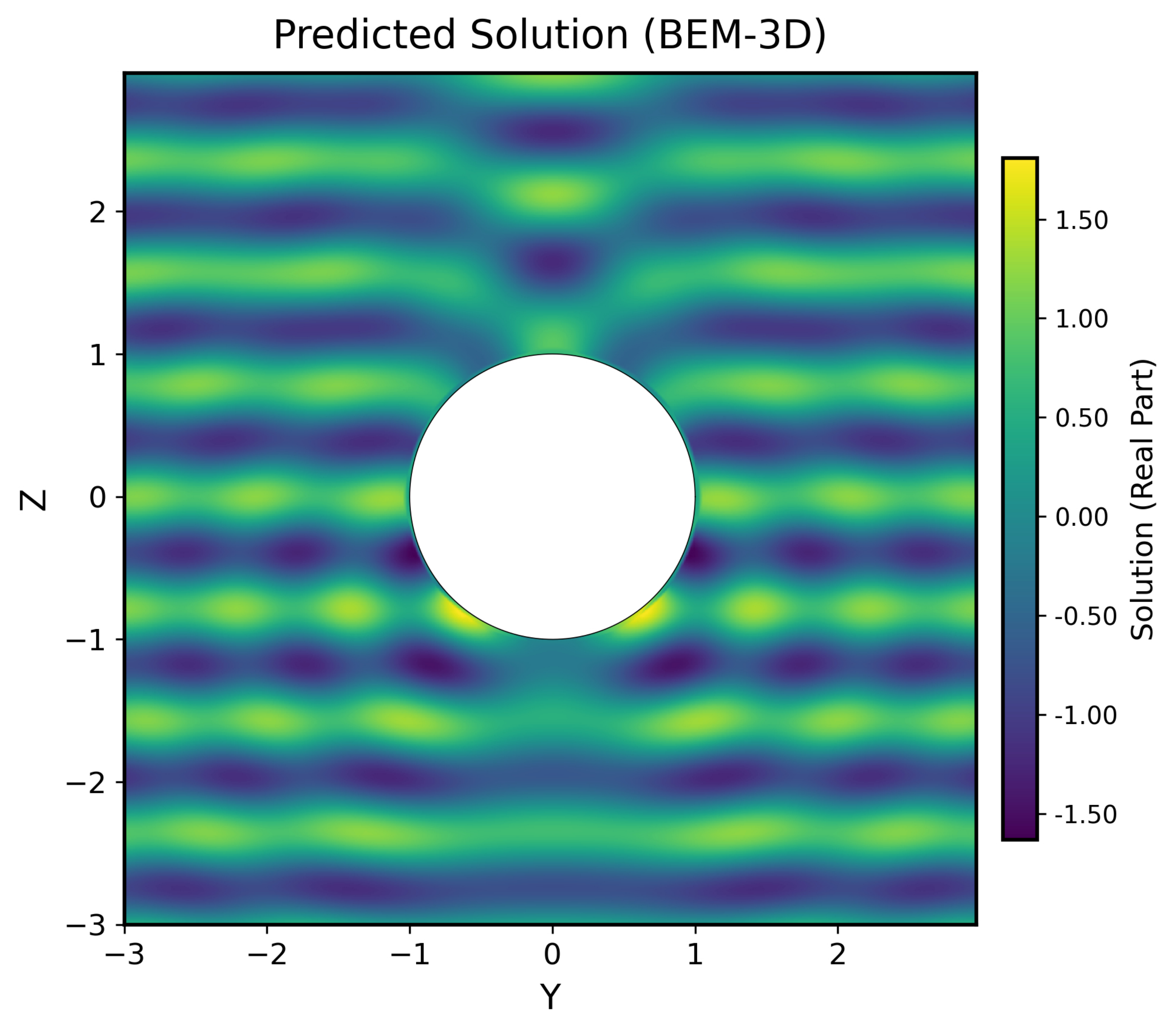}
    \caption{BEM on the Y-Z plane}
    \label{subfig:BEM_k2}
\end{subfigure}
\hfill
\begin{subfigure}{0.3\textwidth}
    \centering
    \includegraphics[width=\linewidth]{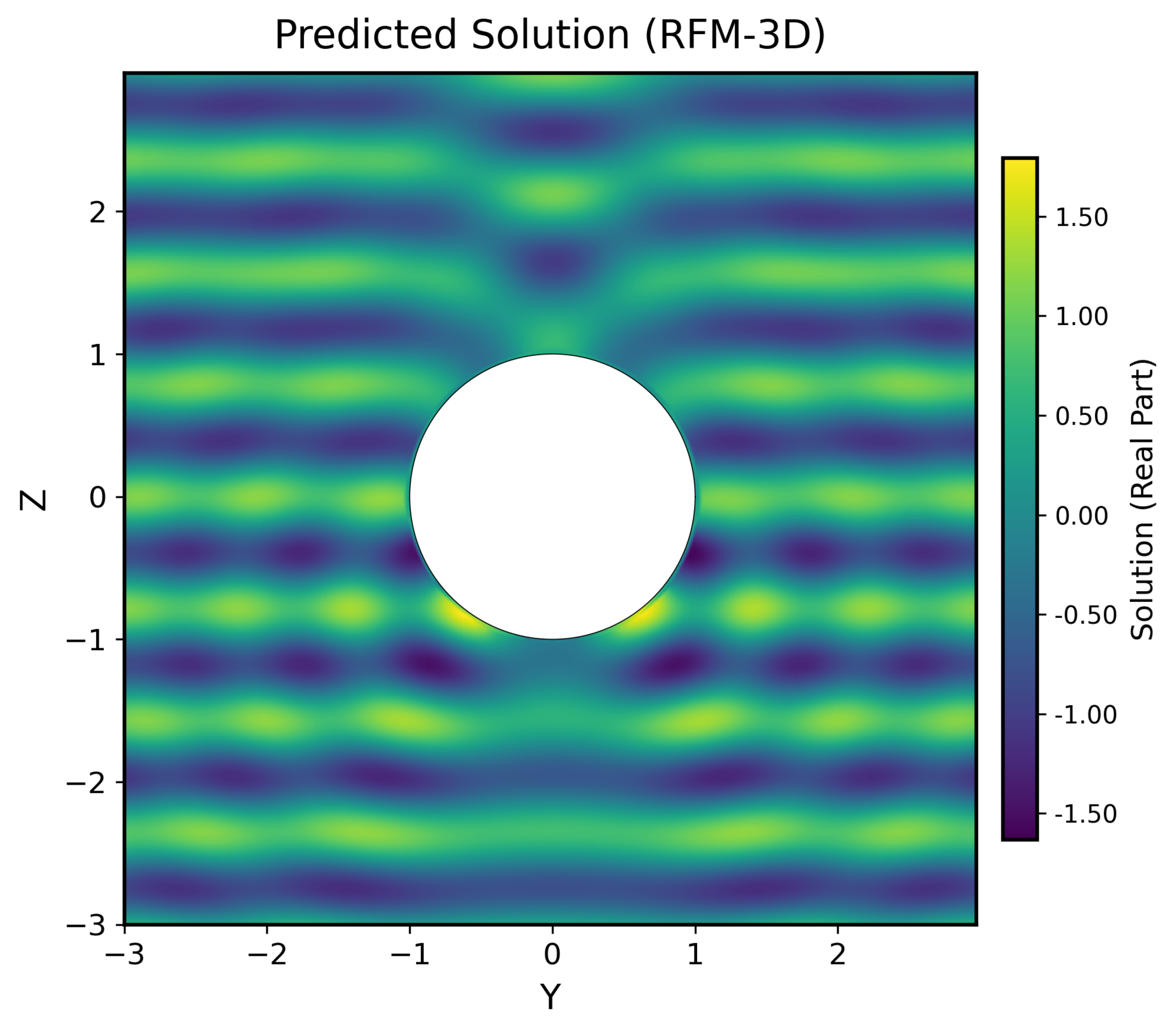}
    \caption{BNM-RF on the Y-Z plane}
    \label{subfig:BNM-RF_k2}
\end{subfigure}

\vspace{0.3cm}

\begin{subfigure}{0.3\textwidth}
    \centering
    \includegraphics[width=\linewidth]{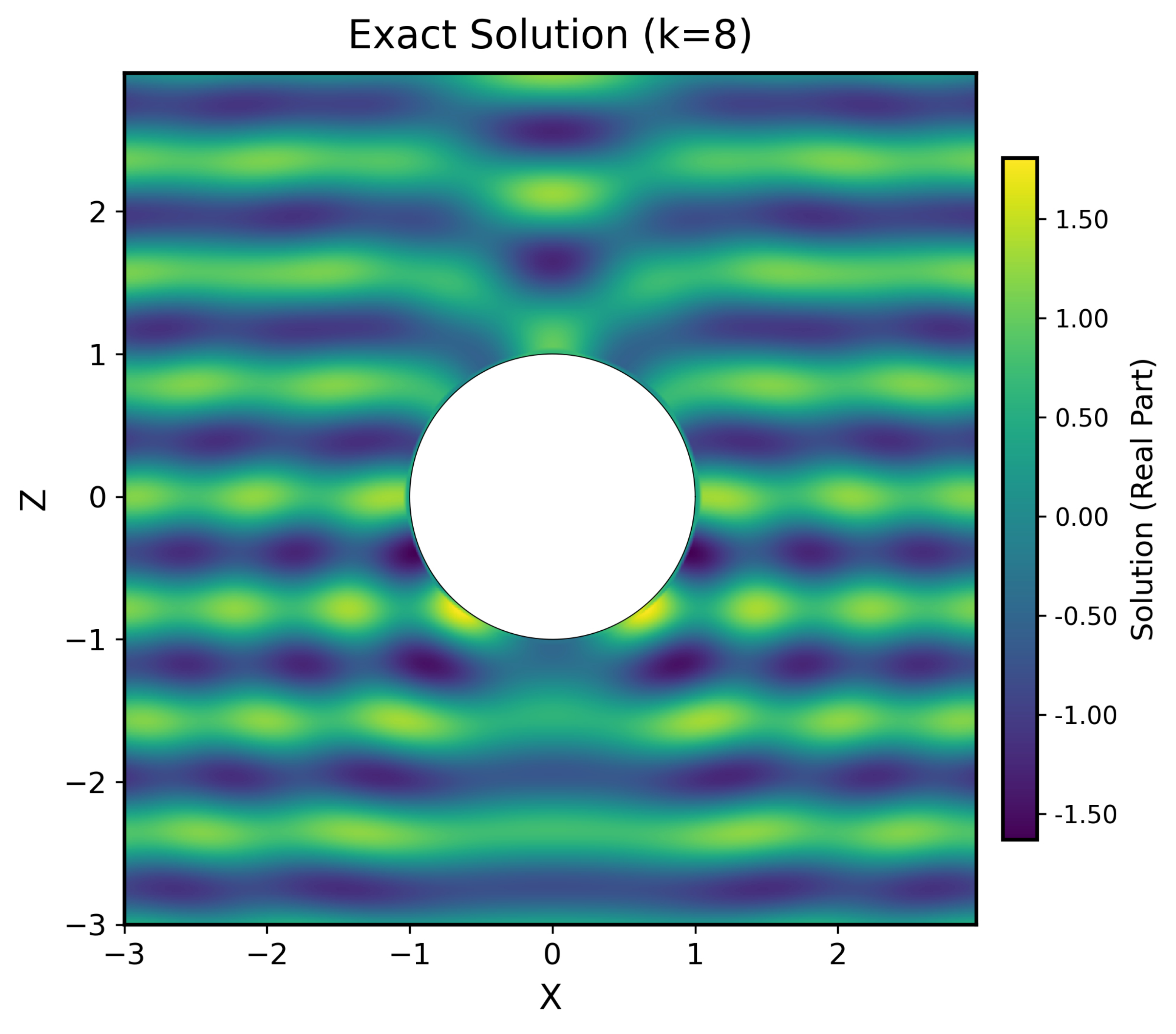}
    \caption{Exact solution on the X-Z plane}
    \label{subfig:BEM_k4}
\end{subfigure}
\hfill
\begin{subfigure}{0.3\textwidth}
    \centering
    \includegraphics[width=\linewidth]{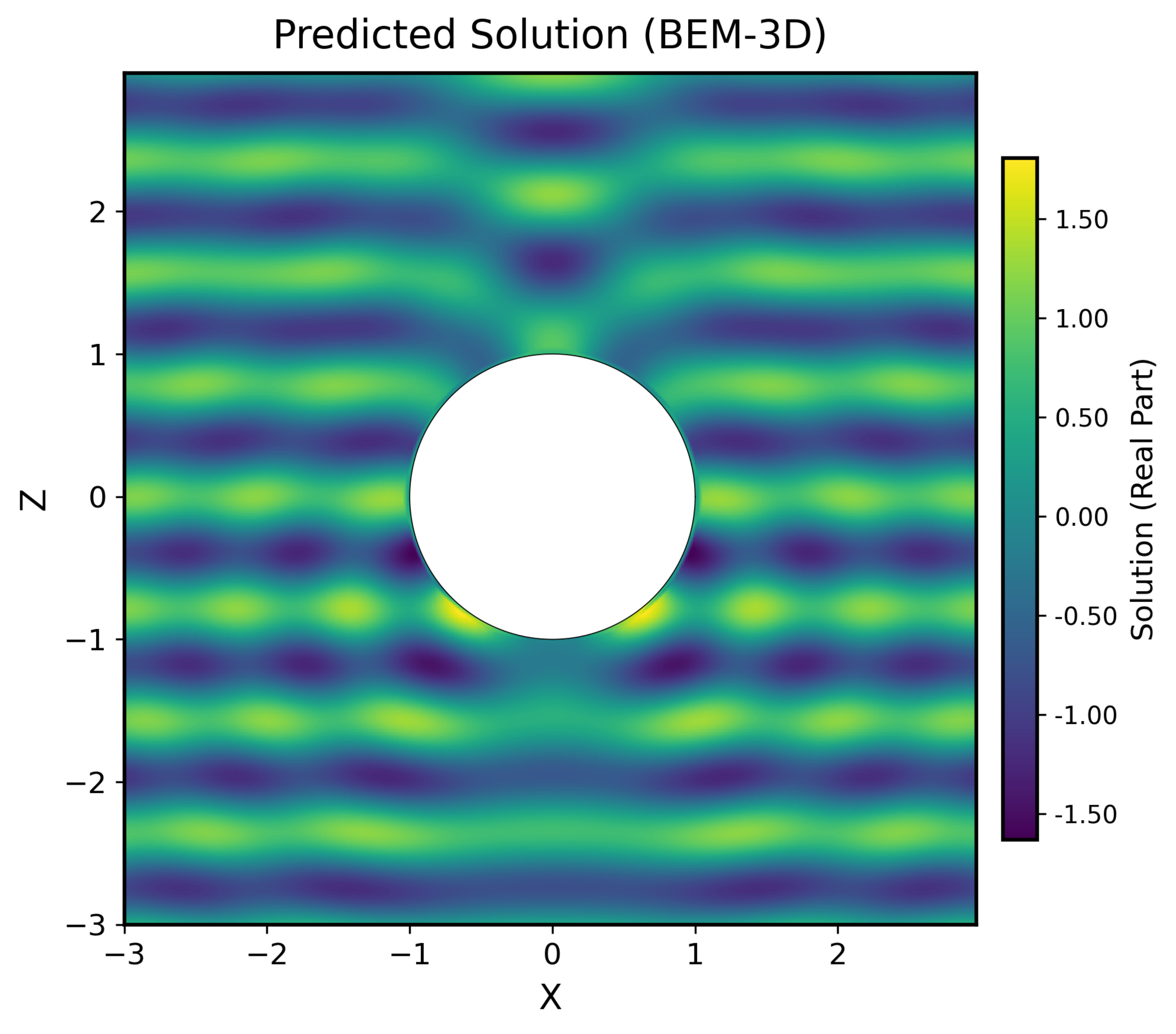}
    \caption{BEM on the X-Z plane}
    \label{subfig:BNM-RF_k4}
\end{subfigure}
\hfill
\begin{subfigure}{0.3\textwidth}
    \centering
    \includegraphics[width=\linewidth]{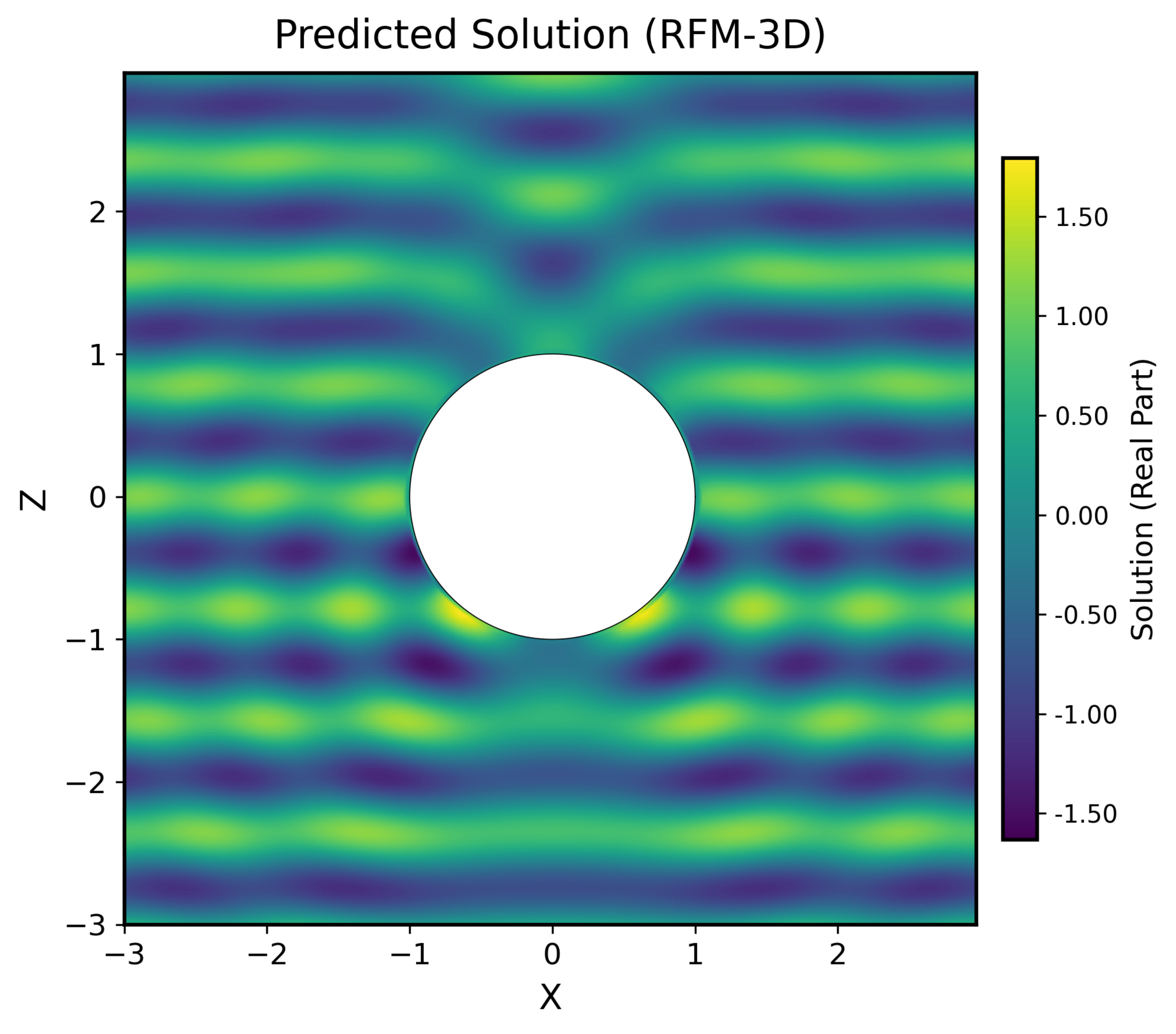}
    \caption{BNM-RF on the X-Z plane}
    \label{subfig:method3_k4}
\end{subfigure}

\vspace{0.3cm}

\begin{subfigure}{0.3\textwidth}
    \centering
    \includegraphics[width=\linewidth]{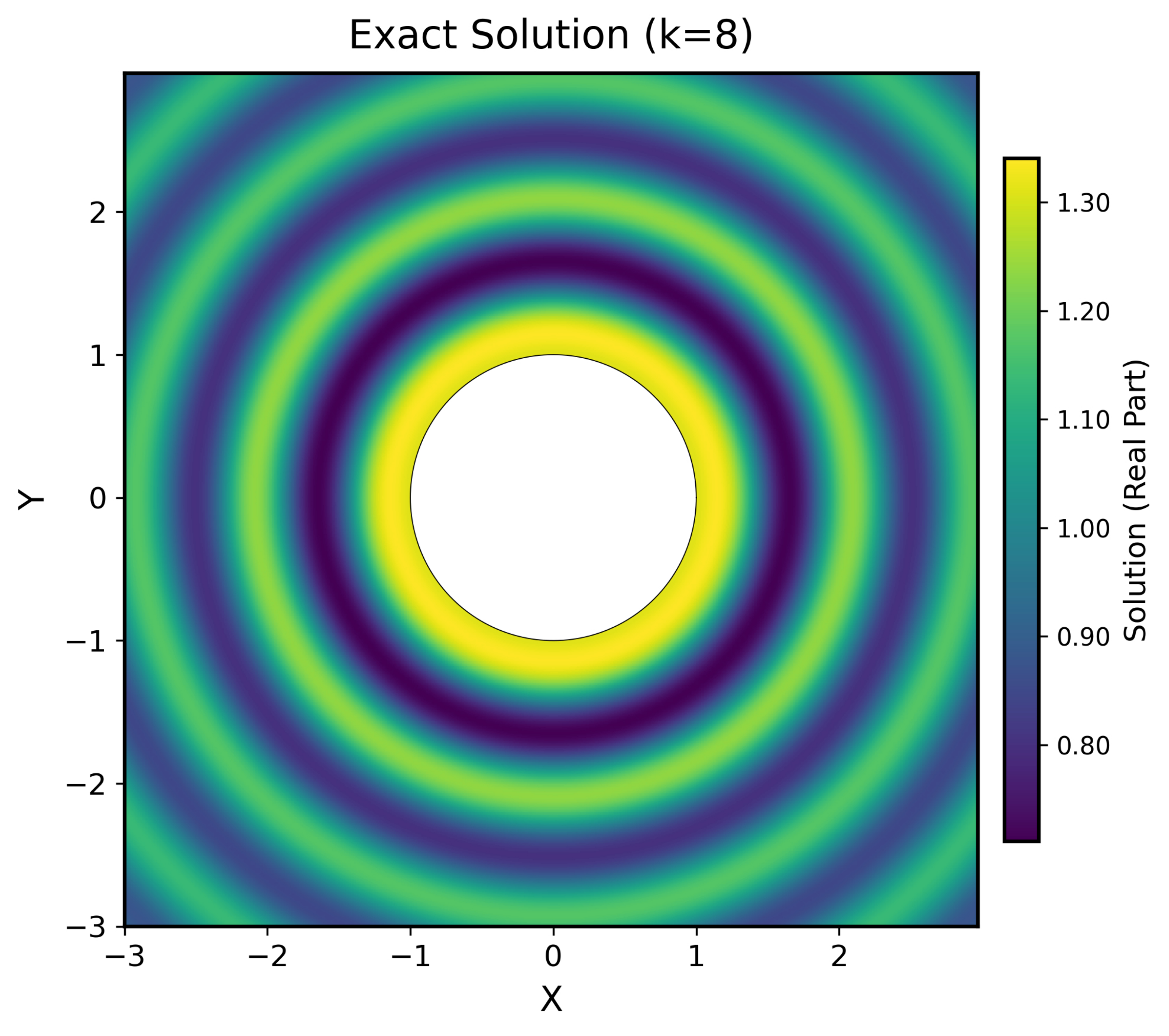}
    \caption{Exact solution on the X-Y plane}
    \label{subfig:BEM_k8}
\end{subfigure}
\hfill
\begin{subfigure}{0.3\textwidth}
    \centering
    \includegraphics[width=\linewidth]{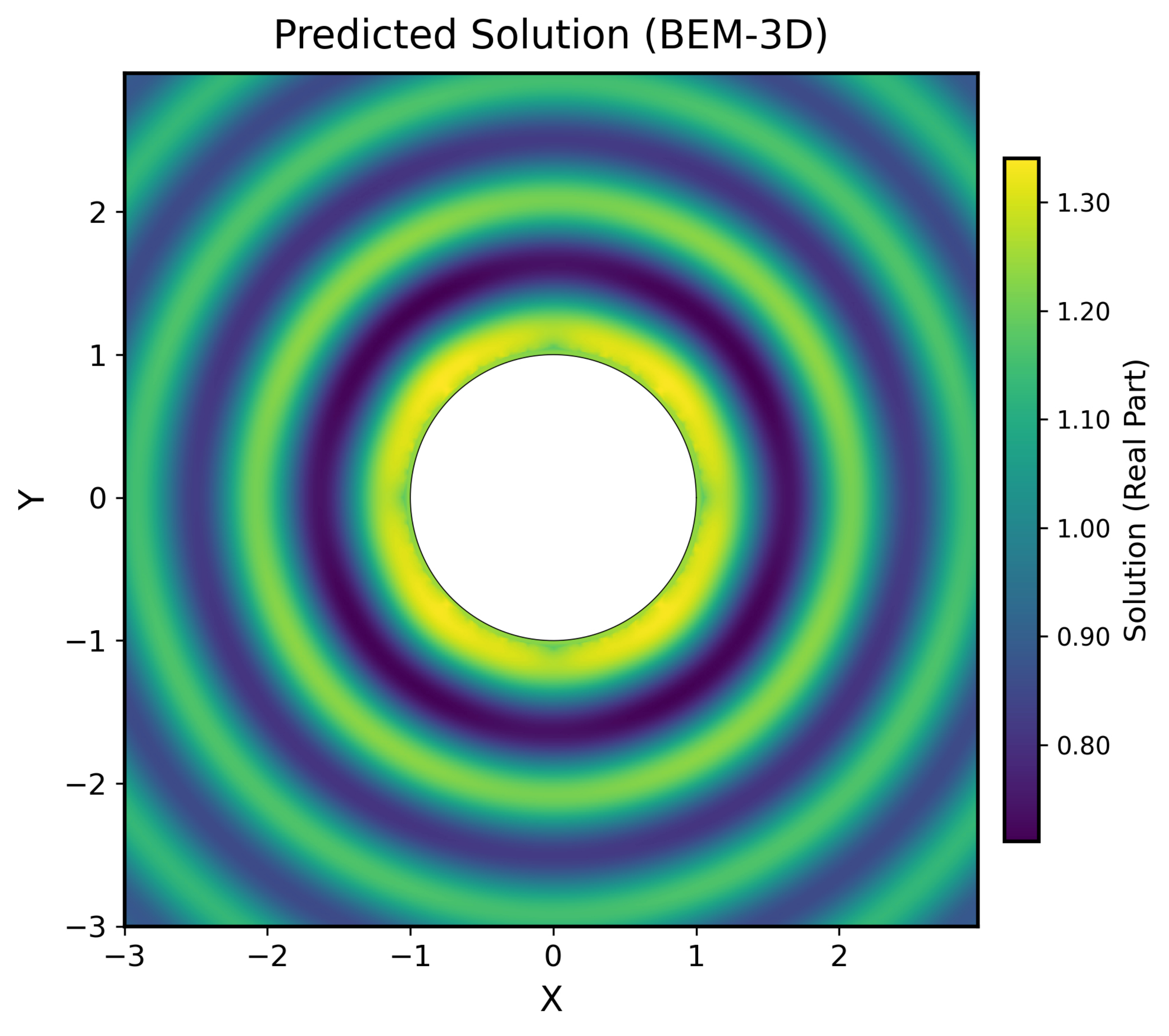}
    \caption{BEM on the X-Y plane}
    \label{subfig:BNM-RF_k8}
\end{subfigure}
\hfill
\begin{subfigure}{0.3\textwidth}
    \centering
    \includegraphics[width=\linewidth]{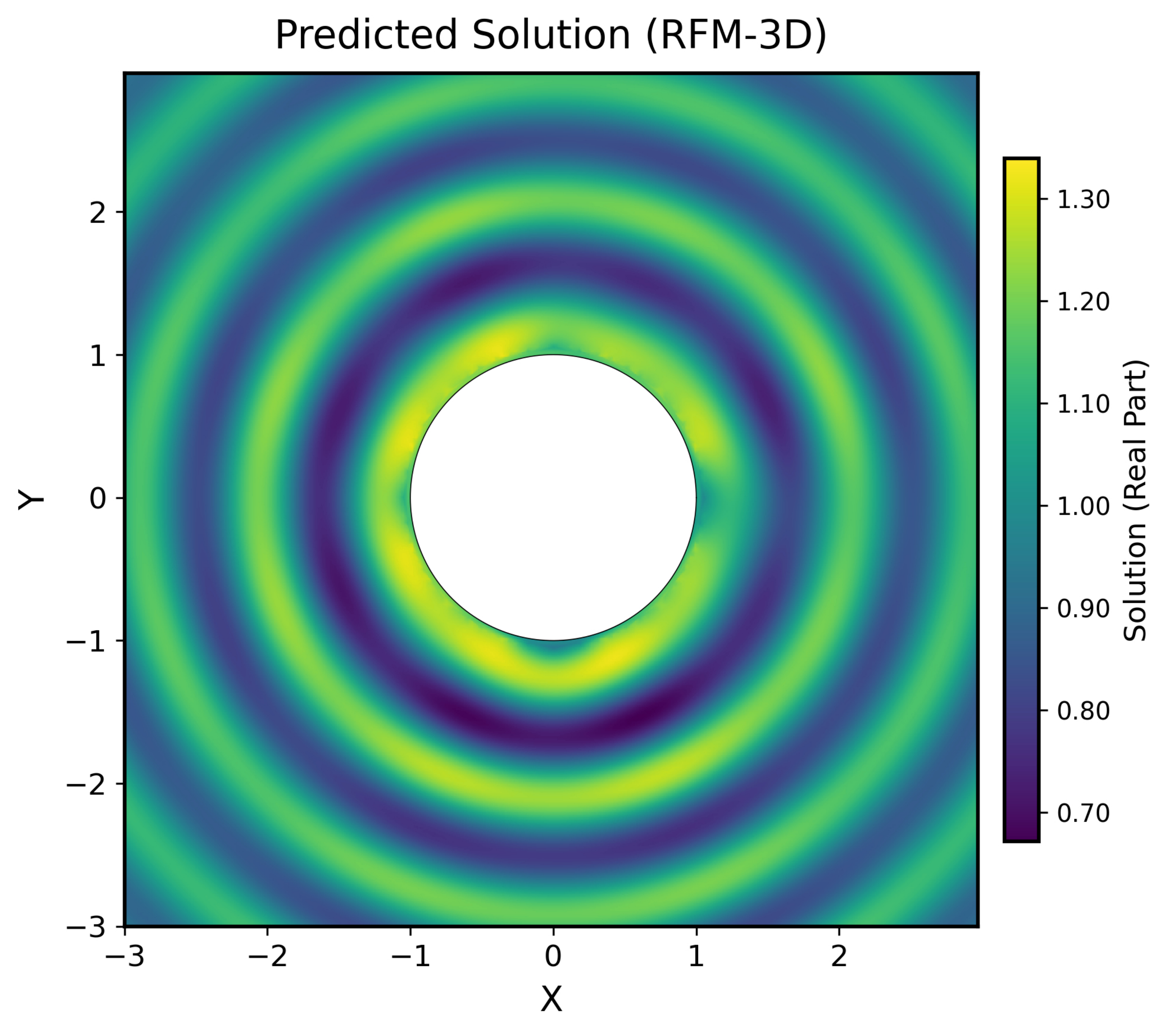}
    \caption{BNM-RF on the X-Y plane}
    \label{subfig:method3_k8}
\end{subfigure}

\caption{Three-dimensional representations of numerical solutions on cross-sectional planes (Y-Z, X-Z, and X-Y) for plane wave scattering by a rigid sphere with wavenumber $k=2$ and $M=32$ neurons, obtained by BEM and BNM-RF.}
\label{fig:plane-wave-impinged}
\end{figure}

\begin{table}[htbp]
  \centering
  \caption{Relative $L^2$ error of plane wave impinged on a sphere with BEM and BNM-RF method}
  \label{tab:error_comparison_scatter}
  \begin{tabular}{lccccc}
    \toprule
     & BEM & BNM-RF(32) & BNM-RF(64)  \\
    \midrule
    $k=4$ & $2.56\times 10^{-2}$ & $3.86\times 10^{-2}$ & $5.45\times 10^{-2}$  \\
    $k=8$ &  $3.95\times 10^{-2}$ & $6.16\times 10^{-2}$ & $7.19\times 10^{-2}$  \\
    $k=12$ & $8.27\times 10^{-2}$ & $1.16\times 10^{-1}$ & $1.40\times 10^{-1}$ \\
    \bottomrule
  \end{tabular}
\end{table}
  
\section{Conclusion and Future Work}
\label{sec:conclusion}

In this study, we proposed BNM-RF for solving boundary value problems through boundary integral equations. The proposed method provides a new way to combine the random feature method with boundary integral formulations. BNM-RF bridges the classical boundary element method and existing random-feature-based approaches for boundary integral equations. It retains the dimensional reduction of boundary integral formulations and is naturally suited to exterior problems, while offering a neural representation of the boundary unknown.

For elliptic problems, we established convergence and obtained error bounds for both the boundary function and the interior solution. Numerical experiments on Laplace and Helmholtz problems, including interior and exterior cases, show that the proposed method is effective in the tested settings. In the two-dimensional examples, it achieves better accuracy than the classical BEM and BINN. In the three-dimensional acoustic scattering task, BNM-RF achieves comparable accuracy to BEM while using only a small number of neurons. The results further suggest that refinements in the treatment of singular integrals and geometric approximation will be important for improving accuracy in more challenging settings.

Therefore, the present work should be viewed as a practical and theoretically grounded framework rather than as a replacement for mature boundary element solvers in all situations. Several directions deserve further study. The current theory is developed for elliptic problems, and broader applications will still depend on the availability of suitable boundary integral formulations and fundamental solutions. Improving the treatment of singular and nearly singular integrals, together with geometric approximation in three dimensions, is likely to be essential for further gains in accuracy. It is also important to analyze the dense matrix arising in the solution process and to combine the present framework with fast multipole methods or other accelerated boundary integral solvers. Finally, extending the framework to time-dependent, variable-coefficient, and nonlinear problems remains an important direction whenever appropriate boundary integral representations can be constructed.

\appendix
\section{Sobolev spaces and sampling inequalities on manifolds}
\label{app1}
\section*{A.1 Basic concepts and definitions}

This appendix introduces the theory of Sobolev spaces defined on smooth compact Riemannian manifolds with corners and presents a key sampling inequality. This inequality serves as the theoretical foundation for analyzing the convergence of collocation methods.

Consider a $t$-dimensional (with $t \leq d$) smooth compact Riemannian manifold with corners, denoted $\mathcal{M}$, embedded in $\mathbb{R}^{d}$. On such a manifold, we can define its intrinsic geodesic distance:
$$
\rho_{\mathcal{M}}(\mathbf{x}, \mathbf{y}):=\inf _{\ell} \int_{0}^{1}\|{\frac{d\ell(t)}{dt}}\| d t, \quad \mathbf{x}, \mathbf{y} \in \mathcal{M},
$$
where the infimum is taken over all piecewise smooth curves $\ell:[0,1]\to\mathcal{M}$ connecting $\mathbf{x}$ and $\mathbf{y}$.

To define Sobolev spaces on the manifold, we use local coordinate charts together with a partition of unity. Let \(\{(M_j,\Psi_j)\}_{j=1}^{N_1}\) be an atlas of \(\mathcal M\), and let \(\{\kappa_j\}\) be a partition of unity subordinate to this atlas. For a function \(f:\mathcal M\to\mathbb R\), define
$$
\pi_j(f)(\mathbf y):=
\begin{cases}
\kappa_j\bigl(f(\Psi_j^{-1}(\mathbf y))\bigr), & \mathbf y\in \Psi_j(M_j),\\[4pt]
0, & \text{otherwise}.
\end{cases}
$$
Here, the set \(\Xi_j\) depends on the type of coordinate chart \(\Psi_j\): it may be \(\mathbb R^t\) (for an interior chart), a half-space (for a boundary chart), or a quadrant (for a corner chart).

Then, for \(u:\mathcal M\to\mathbb R\), the Sobolev norm on $\mathcal M$ is defined by
$$
\|u\|_{H^s(\mathcal M)}
:=
\left(
\sum_{j=1}^{N_1}
\|\pi_j(u)\|_{H^s(\Xi_j)}^2
\right)^{1/2}.
$$
All functions satisfying $\|u\|_{H^{s}(\mathcal{M})}<\infty$ form the Sobolev space $H^{s}(\mathcal{M})$

\section*{A.2 Sampling inequality}

Based on the above setup, we have the following key theorem. This theorem was proven for manifolds without boundary in \cite{fuselier2012scattered}. The main idea of the proof—localizing via the atlas and applying classical sampling theorems—can be directly extended to the case of manifolds with corners.

\begin{Proposition}[Sobolev sampling inequality on manifolds \cite{fuselier2012scattered}]
\label{prop:A1}
Suppose $\mathcal{M} \subset \mathbb{R}^{d}$ is a smooth, compact, Riemannian manifold with corners, of dimension $t$ and let $s > \frac{t}{2}$ and $r \in \mathbb{N}$ satisfy $0 \leq r \leq \lceil s \rceil - 1$. Let $X \subset \mathcal{M}$ be a discrete set with mesh norm $h_{\mathcal{M}, X}$ defined as
$$
h_{\mathcal{M}, X}:=\sup _{\mathbf{x}^{\prime} \in \mathcal{M}} \min _{\mathbf{x} \in X} \rho_{\mathcal{M}}\left(\mathbf{x}^{\prime}, \mathbf{x}\right).
$$
Then there is a constant $h_{0} > 0$ depending only on $\mathcal{M}$ such that if $h_{\mathcal{M}, X} < h_{0}$ and if $u \in H^{s}(\mathcal{M})$ satisfies $\left.u\right|_{X} = 0$ then
$$
\|u\|_{H^{r}(\mathcal{M})} \leq C h_{\mathcal{M}, X}^{s-r} \|u\|_{H^{s}(\mathcal{M})}.
$$
Here $C > 0$ is a constant independent of $h_{\mathcal{M}}$ and $u$.
\end{Proposition}

\section{Abstract framework for convergence rate analysis}
This appendix presents an existing general theoretical framework for analyzing the convergence rates of approximate solutions to nonlinear operator equations.

Consider an operator equation of the form
\begin{equation}\label{eq:operator_equation}
\mathcal{G}(v^*) = w^*,
\end{equation}
where $v^*$ and $w^*$ are elements of the appropriate Banach spaces and $\mathcal{G}$ is a nonlinear operator. In the context of partial differential equations, $\mathcal{G}$ typically represents the differential operator defining the problem, $v^*$ corresponds to the exact solution, and $w^*$ represents the given data . The objective is to approximate $v^*$ under suitable assumptions concerning its regularity and the stability properties of the operator $\mathcal{G}$.

Throughout this section, for a Banach space $V$, we denote by $B_r(V)$ the closed ball of radius $r > 0$ centered at the origin in $V$.

\begin{theorem}[Theorem 3.1,\cite{batlle2025error}]
\label{thm:abstract_banach_spaces}
Consider abstract Banach spaces $(V_i, \|\cdot\|_i)_{i=1}^4$ as well as RKHS $(\mathcal{H}, \|\cdot\|_\mathcal{H})$. Suppose the following conditions are satisfied for any choice of $r>0$ (all the appeared constants $C(r)$ are non-decreasing regarding $r$):
\begin{enumerate}
    \item[(A1)] For any pair $v, v' \in B_r(V_1)$ there exists a constant $C = C(r) > 0$ so that
    $$
    \|v - v'\|_1 \le C \|\mathcal{G}(v) - \mathcal{G}(v')\|_2. 
    $$
    \item[(A2)] For any pair $v, v' \in B_r(V_4)$ there exists a constant $C = C(r) > 0$ so that
    $$
    \|\mathcal{G}(v) - \mathcal{G}(v')\|_3 \le C \|v - v'\|_4. 
    $$
\end{enumerate}
Suppose further that the following hold:
\begin{enumerate}
    \item[(A3)] For any $v \in V_4$, there exists a constant $C > 0$ so that
    $$
    \|v\|_4 \le C \|v\|_\mathcal{H}.
    $$
    \item[(A4)] There exists a set $\widetilde{V} \subset V_2 \cap V_3$ and a constant $\varepsilon > 0$, so that for all $w, w' \in \widetilde{V}$ it holds that
    $$
    \|w - w'\|_2 \le \varepsilon \|w - w'\|_3.
    $$
\end{enumerate}
Suppose problem  is uniquely solvable with $v^* \in \mathcal{H}$ and let $v^\dagger \in \mathcal{H}$ be any other function such that:
\begin{enumerate}
    \item[(A5)] $\mathcal{G}(v^*), \mathcal{G}(v^\dagger) \in \widetilde{V}$.
    \item[(A6)] There exists a constant $C > 0$, independent of $v^*$ and $v^\dagger$, so that
    $$\|v^\dagger\|_\mathcal{H} \le C \|v^*\|_\mathcal{H}.
    $$
\end{enumerate}
Then there exists a constant $C > 0$, depending only on $\|v^*\|_\mathcal{H}$, such that
$$
\|v^\dagger - v^*\|_1 \le C \varepsilon \|v^*\|_\mathcal{H}.
$$
\end{theorem}

\section{Burton-Miller formulation and implementation details}
\label{app:operators}
To solve the exterior acoustic scattering problem, we employ the Burton--Miller formulation to avoid the non-uniqueness issue of the conventional boundary integral equation at characteristic wavenumbers. 

For the numerical implementation in Section~\ref{sec:acoustic-scattering}, the spherical boundary is discretized using a triangular surface mesh generated by Gmsh, resulting in 648 triangular elements. In BNM-RF, the coordinates of the mesh nodes are used as inputs to the neural network. The centroid of each triangular element is chosen as the collocation point. Surface integrals over triangular elements are evaluated using Gaussian quadrature rules \cite{kirkup2019acoustics}, with 7 Gaussian points and the corresponding weights on each element.

The four Helmholtz boundary integral operators are defined as:
\begin{align}
\{L_k \zeta\}_{\Gamma}(\mathbf{p}) &= \int_{\Gamma} G_k(\mathbf{p},\mathbf{q}) \zeta(\mathbf{q}) \, dS_q, \label{eq:C1} \\
\{M_k \zeta\}_{\Gamma}(\mathbf{p}) &= \int_{\Gamma} \frac{\partial G_k}{\partial n_q}(\mathbf{p},\mathbf{q}) \zeta(\mathbf{q}) \, dS_q, \label{eq:C2} \\
\{M_k^t \zeta\}_{\Gamma}(\mathbf{p}; \mathbf{u}_p) &= \frac{\partial}{\partial \mathbf{u}_p} \int_{\Gamma} G_k(\mathbf{p},\mathbf{q}) \zeta(\mathbf{q}) \, dS_q, \label{eq:C3} \\
\{N_k \zeta\}_{\Gamma}(\mathbf{p}; \mathbf{u}_p) &= \frac{\partial}{\partial \mathbf{u}_p} \int_{\Gamma} \frac{\partial G_k}{\partial n_q}(\mathbf{p},\mathbf{q}) \zeta(\mathbf{q}) \, dS_q, \label{eq:C4}
\end{align}
where $\Gamma$ is a boundary or partial boundary, $\mathbf{n}_q$ and $\mathbf{u}_p$ are unit vectors with $\mathbf{n}_q$ the unique normal to $\Gamma$ at $\mathbf{q}$, and $\zeta(\mathbf{q})$ is a function defined for $\mathbf{q} \in \Gamma$. $G_k(\mathbf{p},\mathbf{q})$ is the free-space Green's function for the Helmholtz equation.

Let $\varphi$ denote the potential and $v = \frac{\partial \varphi}{\partial \mathbf{n}}$ its normal derivative on the boundary. The Burton–Miller formulation leads to the following discretized system for collocation points:
\begin{equation}
\bigl[ M_k - \tfrac{1}{2}I + \mu N_k \bigr] \varphi 
= -\varphi^{\mathrm{inc}} - \mu v^{\mathrm{inc}} 
+ \bigl[ L_k + \mu (M_k^t + \tfrac{1}{2}I) \bigr] v, \label{eq:C5}
\end{equation}
where $\mu \in \mathbb{C} \setminus \{0\}$ is a coupling parameter (typically $\mu = \frac{\mathrm{i}}{k+1}$), and $\varphi^{\mathrm{inc}}, v^{\mathrm{inc}}$ are incident fields.

\bibliographystyle{elsarticle-num}
\bibliography{rfm-bie}

\end{document}